\documentclass[final,onefignum,onetabnum]{siamart190516}
\usepackage{amsmath}
\usepackage{amssymb}
\usepackage{url}
\usepackage{color}
\usepackage{array}
\usepackage{colortbl}
\usepackage{blkarray}
\usepackage{caption}
\usepackage{capt-of}
\usepackage{colortbl}
\usepackage{fancyhdr}
\usepackage{fancyvrb}
\usepackage{float}
\usepackage[T1]{fontenc}
\usepackage{framed}
\usepackage{geometry}
\usepackage{graphicx}
\usepackage{hyperref}
\usepackage[latin1]{inputenc}
\usepackage{scalerel}
\usepackage{stmaryrd}
\usepackage{subfigure}
\usepackage{upquote}
\usepackage{xspace}
\usepackage{todonotes}
\usepackage{booktabs}
\usepackage{multirow}
\definecolor{lightgray}{gray}{0.5}
\definecolor{darkgreen}{rgb}{0,0.6,0.13}
\newcommand{\nc}{\newcommand}
\nc{\dsp}{\displaystyle}
\nc{\txt}{\textstyle}
\nc{\creff}[1]{(\cref{#1})}
\nc{\mrm}[1]{\mathrm{#1}}
\nc{\udl}[1]{\underline{#1}}
\nc{\ovl}[1]{\overline{#1}}
\nc{\al}{\underline{\boldsymbol{\alpha}}}
\nc{\la}{\underline{\boldsymbol{\lambda}}}
\nc{\llbr}{\llbracket}
\nc{\rrbr}{\rrbracket}
\nc{\lbr}{\lbrack}
\nc{\rbr}{\rbrack}
\nc{\N}{\mathbb{N}}
\nc{\Z}{\mathbb{Z}}
\nc{\D}{\mathbb{D}}
\nc{\Q}{\mathbb{Q}}
\nc{\R}{\mathbb{R}}
\nc{\C}{\mathbb{C}}
\nc{\T}{\mathbb{T}}
\nc{\Stwo}{\mathbb{S}^2}
\nc{\tld}[1]{\tilde{#1}}
\nc{\wtld}[1]{\widetilde{#1}}
\nc{\hu}{\hat{u}}
\nc{\wh}[1]{\widehat{#1}}
\nc{\Fbf}{\textbf{F}}
\nc{\Gbf}{\textbf{G}}
\nc{\Lbf}{\textbf{L}}
\nc{\Nbf}{\textbf{N}}
\nc{\Ibf}{\textbf{I}}
\nc{\Dbf}{\textbf{D}} 
\nc{\Tbf}{\textbf{T}}
\nc{\Jbf}{\textbf{J}} 
\nc{\Rbf}{\textbf{R}}   
\nc{\ph}{\varphi}
\nc{\NN}{\mathcal{NN}}
\nc{\OO}{\mathcal{O}}
\nc{\sumeven}{\sum_{k=-N/2}^{N/2}{\hspace{-0.3cm}}'{\;\,}}
\nc{\sumevenk}{\sum_{k=-n/2}^{n/2}{\hspace{-0.3cm}}'{\;\,}}
\nc{\sumevenj}{\sum_{j=-m/2}^{m/2}{\hspace{-0.3cm}}'{\;\,}}
\nc{\sumodd}{\sum_{k=-\frac{N-1}{2}}^{\frac{N-1}{2}}}
\nc{\sumoddl}{\sum_{l=-\frac{N-1}{2}}^{\frac{N-1}{2}}}
\nc{\cqfd}{~\hbox{\vrule width 2.5pt depth 2.5 pt height 3.5 pt}}
\nc{\ra}[1]{\renewcommand{\arraystretch}{#1}}
\nc{\bs}[1]{\boldsymbol{#1}}
\nc{\rr}[1]{\textcolor{red}{#1}}
\nc{\eqvsp}{\\[0.25em]}
\nc{\eqvspp}{\\[0.5em]}
\nc{\eqVsp}{\\[0.75em]}
\nc{\eqVspp}{\\[1em]}
\nc{\dmu}{\delta\hspace{-0.025cm}\mu\hspace{0.02cm}}
\nc{\dnu}{\delta\hspace{-0.010cm}\nu\hspace{0.01cm}}
\nc{\dep}{\delta\hspace{-0.015cm}\epsilon\hspace{0.02cm}}
\title{Computing weakly singular and near-singular integrals over curved boundary elements}
\author{Hadrien Montanelli\thanks{INRIA Saclay \& \'{E}cole Polytechnique, 91120 Palaiseau, France. This publication was based on work supported by the Direction G\'en\'erale de l'Armement (grant number AID 2018 60 0074).} 
\and Matthieu Aussal\thanks{INRIA Saclay \& \'{E}cole Polytechnique, 91120 Palaiseau, France. Also supported by grant AID 2018 60 0074.}
\and Houssem Haddar\thanks{INRIA Saclay \& \'{E}cole Polytechnique, 91120 Palaiseau, France.}}

\begin{document}
\setlength{\parskip}{0pt}

\maketitle

\begin{abstract}
We present algorithms for computing weakly singular and near-singular integrals arising when solving the 3D Helmholtz equation with curved boundary elements. These are based on the computation of the preimage of the singularity in the reference element's space using Newton's method, singularity subtraction, the continuation approach, and transplanted Gauss quadrature. We demonstrate the accuracy of our method for quadratic basis functions and quadratic triangles with several numerical experiments, including the scattering by two half-spheres.
\end{abstract}

\begin{keywords}
Helmholtz equation, integral equations, boundary element method, singular integrals, near-singular integrals, Taylor series, homogeneous functions, continuation approach, Gauss quadrature
\end{keywords}

\begin{AMS}
35J05, 41A55, 41A58, 45E05, 45E99, 65N38, 65R20, 78M15
\end{AMS}

\pagestyle{myheadings}
\thispagestyle{plain}

\markboth{MONTANELLI, AUSSAL AND HADDAR}{Computing weakly singular and near-singular integrals in high-order boundary elements}

\section{Introduction}

The Helmholtz equation $\Delta u + k^2u = 0$ in the presence of an obstacle may be rewritten as an integral equation on the obstacle's boundary via \textit{layer potentials}~\cite{colton1983}. For example, the radiating solution to the Dirichlet problem $\Delta u + k^2u = 0$ in $\R^3\setminus\overline{\Omega}$ with $u = u_D$ on $\Gamma=\partial\Omega$, for some bounded $\Omega$ whose complement is connected, can be obtained via the equation \cite[Thm.~3.28]{colton1983}
\begin{align}
\int_{\Gamma}G(\bs{x},\bs{y})\varphi(\bs{y})d\Gamma(\bs{y}) = u_D(\bs{x}), \quad \bs{x}\in\Gamma,
\label{eq:SL}
\end{align}
based on the \textit{single-layer potential}; $G$ is the Green's function of the Helmholtz equation in 3D,
\begin{align}
G(\bs{x},\bs{y}) = \frac{1}{4\pi}\frac{e^{ik\vert\bs{x}-\bs{y}\vert}}{\vert\bs{x}-\bs{y}\vert}.
\label{eq:kernel}
\end{align}
Once \cref{eq:SL} is solved for $\varphi$, unique if $k^2$ is not a Dirichlet eigenvalue of $-\Delta$ in $\Omega$ \cite[Thm.~3.30]{colton1983}, the solution $u$ may be represented by the left-hand side of \cref{eq:SL} for all $\bs{x}\in\R^3\setminus\Omega$.

From a numerical point of view, integral equations of the form of \cref{eq:SL} are particularly challenging for several reasons. First, when $\bs{x}$ approaches $\bs{y}$, the integral becomes singular and standard quadrature schemes fail to be accurate---analytic integration or carefully-derived quadrature formulas are, hence, required. Second, the resulting linear systems after discretization are often dense. For large wavenumbers $k$, only iterative methods can be used to solve them (with the help of specialized techniques to accelerate the matrix-vector products, such as the Fast Multipole Method~\cite{greengard1987} or hierarchical matrices \cite{hackbusch2015}). In this respect, the use of high-order numerical discretization schemes may be helpful in enlarging the interval of feasible wavenumbers.

Two of the most popular methods for solving such equations are the \textit{Nystr\"{o}m} and \textit{boundary element} methods. Nystr\"{o}m methods, which seek the numerical solution of \cref{eq:SL} by replacing the integral with an appropriately weighted sum, exhibit high-order convergence but are often limited to smooth and simple geometries \cite{kress2014}. These are particularly efficient when solving the 2D Helmholtz equation with 1D integrals, for which high-order quadrature rules may be derived \cite{alpert1999, hao2014, kapur1997, klockner2013, kress1991}. For the 3D equation with 2D integrals, there is no simple high-order quadrature rule, which makes it considerably more difficult. There are, nonetheless, some Nystr\"{o}m methods available \cite{bruno2020, bruno2007, bruno2001, perezarancibia2019}.

\begin{table}
\vspace{.5cm}
\caption{\textit{Techniques that have been proposed to deal with the singular and near-singular integrals that arise when using boundary element methods for 3D problems. The method we propose belongs to the bottom-right category and combines the continuation approach with singularity subtraction.}}
\centering
\ra{1.3}
\begin{tabular}{c|cc}
\toprule
& {\small\textbf{Singular Integrals}} & {\small\textbf{Singular \& Near-singular Integrals}} \\
\midrule
\multirow{3}{*}{{\small\textbf{Flat Elements}}} & subtraction -- & subtraction \cite{jarvenpaa2006} \\
& cancellation \cite{duffy1982, reid2015} & cancellation \cite{graglia2008, hackbusch1994, johnston2013, khayat2005, scuderi2008} \\
& imbedding/continuation \cite{rosen1992, vijayakumar1988} & imbedding/continuation \cite{lenoir2012, rosen1993, salles2013, vijayakumar1989} \\
\midrule
\multirow{3}{*}{{\small\textbf{Curved Elements}}} & subtraction \cite{aliabadi1985, guiggiani1990, guiggiani1992, hall1991} & subtraction -- \\
& cancellation \cite{hackbusch1993} & cancellation \cite{hayami1987, ma2002, qin2011, telles1987}\\
& imbedding/continuation -- & imbedding/continuation \cite{rosen1995} \\
\bottomrule
\end{tabular}
\label{table:intsing}
\end{table}

Boundary element methods, which are based on variational formulations of the integral \cref{eq:SL}, are much more flexible with respect to geometry but typically achieve low-order convergence \cite{sauter2011}. Various techniques have been proposed to deal with the singular and near-singular integrals that arise when using these methods for 3D problems, including \textit{singularity subtraction} \cite{aliabadi1985, guiggiani1990, guiggiani1992, hall1991, jarvenpaa2006}, \textit{singularity cancellation} \cite{duffy1982, graglia2008, hackbusch1993, hackbusch1994, hayami1987, johnston2013, khayat2005, ma2002, qin2011, reid2015, scuderi2008, telles1987}, the \textit{invariant imbedding} \cite{rosen1992, vijayakumar1988}, and the \textit{continuation approach} \cite{lenoir2012, rosen1993, rosen1995, salles2013, vijayakumar1989}.  We classify these methods in \cref{table:intsing}, depending on their ability to handle flat/curved elements and singular/near-singular integrals---the bottom right category is the most general and challenging one.

In singularity subtraction schemes, which go back to Aliabadi, Hall, and Phemister in 1985~\cite{aliabadi1985}, terms having the same asymptotic behavior as the integrand at the singularity are first subtracted, leaving a bounded difference that may be integrated numerically. The singular terms, captured via asymptotic expansions, are then integrated analytically in one \cite{guiggiani1990, guiggiani1992} or both variables \cite{aliabadi1985, hall1991, jarvenpaa2006}.

Singularity cancellation schemes, which can be traced back to Duffy's paper in the 1980s \cite{duffy1982}, rely on a change of variables such that the Jacobian of the transformation cancels the singularity. The resulting integrand is then analytic in the transformed variables and hence amenable to numerical integration by a Cartesian product of Gauss quadrature rules. Examples of such transformations include a polar coordinate mapping \cite{hackbusch1994}, modifying the distance function in the integrand \cite{ma2002, qin2011}, and various nonlinear functions \cite{duffy1982, graglia2008, hackbusch1993, hayami1987, johnston2013, khayat2005, reid2015, scuderi2008, telles1987}.

Both the invariant imbedding method and the continuation approach utilize the homogeneity of the integrand to transform 2D singular integrals on boundary elements to regular 1D integrals along their contours. The imbedding method for singular integrals was first introduced by Vijayakumar and Cormac in 1988 \cite{vijayakumar1988}, who subsequently proposed the continuation approach as a generalization to near-singular integrals \cite{vijayakumar1989}. These techniques were further extended in the 1990s \cite{rosen1992, rosen1993, rosen1995}. Note that, in \cite{rosen1995}, the continuation approach was combined with singularity subtraction for high-order curved elements. Finally, Lenoir and Salles proposed in the 2010s a fully analytic method for computing near-singular integrals borrowing ideas from the continuation approach \cite{lenoir2012, salles2013}.

We propose in this paper a novel method for computing singular/near-singular integrals that arise when solving \cref{eq:SL} and evaluating the solution close to $\Gamma$ with curved triangular elements. More specifically, we consider weakly singular and near-singular integrals of the form
\begin{align}
I(\bs{x}_0) = \int_\mathcal{T}\frac{\varphi(F^{-1}(\bs{x}))}{\vert\bs{x} - \bs{x}_0\vert}dS(\bs{x}),
\label{eq:intsing}
\end{align}
where $\mathcal{T}$ is a \textit{curved} triangle defined by a polynomial transformation $F:\widehat{T}\to\mathcal{T}$ of degree $q\geq1$ from some \textit{flat} reference triangle $\widehat{T}$, $\bs{x}_0$ is a point on or close to $\mathcal{T}$, and $\varphi:\widehat{T}\to\R$ is a polynomial function of degree~$p\geq0$ (not necessarily equal to $q$). Our method is based on the computation of the preimage of the singularity in the reference element's space using Newton's method, singularity subtraction with high-order Taylor-like asymptotic expansions, the continuation approach, and transplanted Gauss quadrature. Integrals of the form of \cref{eq:intsing} appear when \textit{evaluating} the solution---we will also look at integrals over two curved triangles, which occur when \textit{solving} the integral equation.

Combining the continuation approach with singularity subtraction was first proposed in \cite{rosen1995}. The originality of our approach is that we first map back to the reference triangle, then perform singularity subtraction, and finally employ the continuation approach. The authors in \cite{rosen1995} directly utilized the continuation approach with singularity subtraction via the introduction of local coordinate systems---these systems must be constructed and stored for each element, which is trickier to implement and computationally more expensive. On top of this, their method does not handle the case where the singularity is close to an edge of the element---the hardest case in practice.

We present our method based on first-order Taylor series in \cref{sec:first-order}, its extension to high-order approximations in \cref{sec:high-order}, and numerical examples in \cref{sec:numerics}, including the numerical solution of the scattering by two half-spheres using quadratic basis functions and quadratic triangular elements.

\section{Method based on first-order Taylor series}\label{sec:first-order}

We expose here our method to compute integrals of the form of \cref{eq:intsing} using first-order Taylor expansions. The extension of our algorithms to higher order approximations will be presented in \cref{sec:high-order}. We proceed in five steps.

\paragraph{Step 1.~Mapping back} We map $\mathcal{T}$ back to the reference element $\widehat{T}$,
\begin{align}
I(\bs{x}_0) = \int_{\widehat{T}}\frac{\psi(\bs{\hat{x}})}{\vert F(\bs{\hat{x}}) - \bs{x}_0\vert}dS(\bs{\hat{x}}),
\label{eq:step1}
\end{align}
where $F:\widehat{T}\to\mathcal{T}$, the transformation of degree $q$ from the 2D flat triangle $\widehat{T}$ to the 3D curved triangle $\mathcal{T}$, has a $3\times2$ Jacobian matrix $J$ with columns $J_1$ and $J_2$, and $\psi(\bs{\hat{x}})=\varphi(\bs{\hat{x}})\vert J_1(\bs{\hat{x}})\times J_2(\bs{\hat{x}})\vert$. We provide an explicit example of such a mapping $F$ for quadratic triangles ($q=2$) in \cref{sec:step1}.

\paragraph{Step 2.~Locating the singularity} We write
\begin{align}
\bs{x}_0 = F(\bs{\hat{x}}_0) + \bs{x}_0 - F(\bs{\hat{x}}_0)
\label{eq:step2}
\end{align}
for some $\bs{\hat{x}}_0\in\R^2$ such that $F(\bs{\hat{x}}_0)\in\mathcal{S}$ is the closest point to $\bs{x}_0$ on the surface $\mathcal{S}$ defined by
\begin{align}
\mathcal{S} = \left\{F(\bs{\hat{x}}), \; \bs{\hat{x}}\in\R^2\right\} \supset \mathcal{T} = \left\{F(\bs{\hat{x}}), \; \bs{\hat{x}}\in\widehat{T}\right\}.
\end{align}
In other words, we compute the preimage $\bs{\hat{x}}_0$ of the singularity or near-singularity; see \cref{sec:step2}.

Let $\rho$ be the diameter of $\mathcal{T}$, defined as the largest (Euclidean) distance between two points~on~$\mathcal{T}$, and let us define the parameter $h\geq0$ of near-singularity and the unit vector $\bs{e}_h$ (for $h\neq0$) via
\begin{align}
h=\vert F(\bs{\hat{x}}_0)-\bs{x}_0\vert, \quad \bs{e}_h = \frac{F(\bs{\hat{x}}_0)-\bs{x}_0}{h}.
\end{align}
The integral \cref{eq:step1} is singular when $h=0$ and $\bs{\hat{x}}_0\in\widehat{T}$, which implies $\bs{x}_0=F(\bs{\hat{x}}_0)\in\mathcal{T}$, and near-singular when $h\ll \rho$ and $\bs{\hat{x}}_0$ is close to the reference element. (When $h\sim\rho$, the integral is analytic in $\bs{\hat{x}}$ and may be computed exponentially accurately with Gauss quadrature on triangles \cite{lether1976}.)

\paragraph{Step 3.~Taylor expanding/subtracting} We compute the singular or near-singular term in \cref{eq:step1} using the first-order Taylor series of $F(\bs{\hat{x}}) - \bs{x}_0$ at $\bs{\hat{x}}_0$ (see \cref{sec:step3} for the detailed calculations),
\begin{align}
T_{-1}(\bs{\hat{x}},h) = \frac{\psi(\bs{\hat{x}}_0)}{\sqrt{\vert J(\bs{\hat{x}}_0)(\bs{\hat{x}} - \bs{\hat{x}}_0)\vert^2 + h^2}},
\label{eq:Tn1}
\end{align}
and add it to/subtract it from \cref{eq:step1},
\begin{align}
I(\bs{x}_0) = \int_{\widehat{T}}T_{-1}(\bs{\hat{x}},h)dS(\bs{\hat{x}}) + \int_{\widehat{T}}\left[\frac{\psi(\bs{\hat{x}})}{\vert F(\bs{\hat{x}}) - \bs{x}_0\vert} - T_{-1}(\bs{\hat{x}},h)\right]dS(\bs{\hat{x}}).
\end{align}
The first integral is singular or near-singular and will be computed in Steps 4--5. The second integral has a bounded integrand---it can be computed with Gauss quadrature on triangles \cite{lether1976}. To render the integrand in the second integral smoother, which would accelerate convergence by quadrature, higher-order Taylor expansions would be needed; this will be discussed in \cref{sec:high-order}.

\paragraph{Step 4.~Continuation approach} Let
\begin{align}
I_{-1}(h) = \int_{\widehat{T}}T_{-1}(\bs{\hat{x}},h)dS(\bs{\hat{x}}) = \int_{\widehat{T}-\bs{\hat{x}}_0}\frac{\psi(\bs{\hat{x}}_0)}{\sqrt{\vert J(\bs{\hat{x}}_0)\bs{\hat{x}}\vert^2 + h^2}}dS(\bs{\hat{x}}).
\label{eq:In1_tri}
\end{align}
The integrand is homogeneous in both $\bs{\hat{x}}$ and $h$, and using the continuation approach \cite{rosen1995}, we reduce the 2D integral \cref{eq:In1_tri} to a sum of three 1D integrals along the edges of the shifted triangle $\widehat{T}-\bs{\hat{x}}_0$,
\begin{align}
I_{-1}(h) = \psi(\bs{\hat{x}}_0)\sum_{j=1}^3\hat{s}_j\int_{\partial\widehat{T}_j - \bs{\hat{x}}_0}\frac{\sqrt{\vert J(\bs{\hat{x}}_0)\bs{\hat{x}}\vert^2+h^2}-h}{\vert J(\bs{\hat{x}}_0)\bs{\hat{x}}\vert^2}ds(\bs{\hat{x}}),
\label{eq:In1_vtx}
\end{align}
where the $\hat{s}_j$'s are the distances from the origin to the edges $\partial\widehat{T}_j-\bs{\hat{x}}_0$ of $\widehat{T}-\bs{\hat{x}}_0$ (see \cref{sec:step4}).

\paragraph{Step 5.~Transplanted Gauss quadrature} On the one hand, when the origin is far from all three edges (\textit{e.g.}, near the center of the triangle), each integrand in \cref{eq:In1_vtx} is analytic---convergence with Gauss quadrature is exponential \cite[Thm.~19.3]{trefethen2019}. On the other, when the origin lies on an edge, the corresponding integrand is singular---however, the distance $\hat{s}_j$ to that edge equals $0$, the product ``$\hat{s}_j$ times integral'' is also~$0$, and the integral need not be computed (see \cref{sec:step5}). Issues arise when the origin is close to one of the edges---the integrand is analytic but near-singular, convergence with Gauss quadrature is exponential but slow. We circumvent the near-singularity issue by using transplanted Gauss quadrature \cite{hale2008}. We take advantage of the \textit{a priori} knowledge of the singularities to utilize transplanted rules with significantly improved convergence rates (\cref{thm:convergence}).

We now expose each step of our method in detail with an emphasis on quadratic triangles.

\subsection{Mapping back (Step 1)}\label{sec:step1}

Let $\widehat{T}$ be the reference triangle,
\begin{align}
\widehat{T} = \{(\hat{x}_1,\hat{x}_2) \, : \, 0\leq\hat{x}_1\leq1, \; 0\leq\hat{x}_2\leq1-\hat{x}_1\} \subset\R^2.
\end{align}
A quadratic triangle $\mathcal{T}\subset\R^3$ is defined by six points $\bs{a}_j\in\R^3$ and the map $F:\widehat{T}\to\mathcal{T}$ given by
\begin{align}
F(\bs{\hat{x}}) = \sum_{j=1}^6\varphi_j(\bs{\hat{x}})\bs{a}_j\in\R^3,
\label{eq:mapping}
\end{align}
where the $\varphi_j$'s are the real-valued quadratic basis functions defined on $\widehat{T}$ by
\begin{align}
& \varphi_1(\bs{\hat{x}}) = \lambda_1(\bs{\hat{x}})(2\lambda_1(\bs{\hat{x}}) - 1), \quad\quad \varphi_4(\bs{\hat{x}}) = 4\lambda_1(\bs{\hat{x}})\lambda_2(\bs{\hat{x}}), \nonumber \\
& \varphi_2(\bs{\hat{x}}) = \lambda_2(\bs{\hat{x}})(2\lambda_2(\bs{\hat{x}}) - 1), \quad\quad \varphi_5(\bs{\hat{x}}) = 4\lambda_2(\bs{\hat{x}})\lambda_3(\bs{\hat{x}}), \label{eq:basis_func} \\
& \varphi_3(\bs{\hat{x}}) = \lambda_3(\bs{\hat{x}})(2\lambda_3(\bs{\hat{x}}) - 1), \quad\quad \varphi_6(\bs{\hat{x}}) = 4\lambda_1(\bs{\hat{x}})\lambda_3(\bs{\hat{x}}), \nonumber
\end{align}
with $\lambda_1(\bs{\hat{x}}) = 1 - \hat{x}_1 - \hat{x}_2$, $\lambda_2(\bs{\hat{x}}) = \hat{x}_1$, and $\lambda_3(\bs{\hat{x}}) = \hat{x}_2$; we show in \cref{fig:tri} an example of such a triangle. The $3\times2$ Jacobian matrix $J$ is then defined by
\begin{align}
J(\bs{\hat{x}}) = 
\left(
\begin{array}{c|c}
& \\
\hspace{-0.15cm} J_1(\bs{\hat{x}}) & J_2(\bs{\hat{x}}) \hspace{-0.15cm}\phantom{} \\
&
\end{array}
\right)
=
\left(
\begin{array}{c|c}
& \\
\hspace{-0.15cm} F_{\hat{x}_1}(\bs{\hat{x}}) & F_{\hat{x}_2}(\bs{\hat{x}}) \hspace{-0.15cm}\phantom{} \\
&
\end{array}
\right)\in\R^{3\times2},
\end{align}
with (componentwise) partial derivatives $F_{\hat{x}_1}(\bs{\hat{x}})\in\R^3$ and $F_{\hat{x}_2}(\bs{\hat{x}})\in\R^3$ with respect to $\hat{x}_1$ and $\hat{x}_2$.

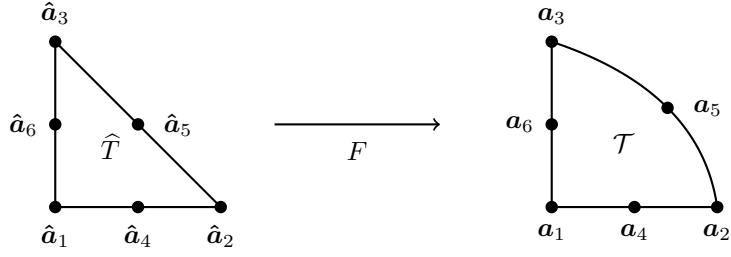
\begin{figure}
\centering
\begin{tikzpicture}[scale = 1.10]
      \draw[thick] (0, 0) -- (2, 0) {};
      \draw[thick] (0, 0) -- (0, 2) {};
      \draw[thick] (2, 0) -- (0, 2) {};
      \node[fill, circle, scale=0.5] at (0, 0) {};
      \node[anchor=north] at (0, 0-0.1) {$\bs{\hat{a}}_1$};
      \node[fill, circle, scale=0.5] at (2, 0) {};
      \node[anchor=north] at (2, 0-0.1) {$\bs{\hat{a}}_2$};
      \node[fill, circle, scale=0.5] at (0, 2) {};
      \node[anchor=south] at (0, 2+0.1) {$\bs{\hat{a}}_3$};
      \node[fill, circle, scale=0.5] at (1, 0) {};
      \node[anchor=north] at (1, 0-0.1) {$\bs{\hat{a}}_4$};
      \node[fill, circle, scale=0.5] at (1, 1) {};
      \node[anchor=west] at (1+0.2, 1) {$\bs{\hat{a}}_5$};
      \node[fill, circle, scale=0.5] at (0, 1) {};
      \node[anchor=east] at (0-0.1, 1) {$\bs{\hat{a}}_6$};
      \node[anchor=north] at (2/3, 3/3) {$\widehat{T}$};
      \draw[thick, ->] (1.65+1, 1) -- (3.65+1, 1) {};
      \node[anchor=north] at (2.65+1, 1-0.1) {$F$};
      \draw[thick] (4+2, 0) -- (6+2, 0) {};
      \draw[thick] (4+2, 0) -- (4+2, 2) {};
      \draw[thick, domain=-1:1, smooth, variable=\theta, black] plot ({6+2*(1-(\theta+1)/2) + 2*2*(2*0.7-1)*(1-(\theta+1)/2)*(\theta+1)/2}, {2*(\theta+1)/2 + 2*2*(2*0.6-1)*(1-(\theta+1)/2)*(\theta+1)/2});
      \node[fill, circle, scale=0.5] at (4+2, 0) {};
      \node[anchor=north] at (4+2, 0-0.1) {$\bs{a}_1$};
      \node[fill, circle, scale=0.5] at (6+2, 0) {};
      \node[anchor=north] at (6+2, 0-0.1) {$\bs{a}_2$};
      \node[fill, circle, scale=0.5] at (4+2, 2) {};
      \node[anchor=south] at (4+2, 2+0.1) {$\bs{a}_3$};
      \node[fill, circle, scale=0.5] at (5+2, 0) {};
      \node[anchor=north] at (5+2, 0-0.1) {$\bs{a}_4$};
      \node[fill, circle, scale=0.5] at (6+2*0.7, 2*0.6) {};
      \node[anchor=west] at (6+2*0.7+0.2, 2*0.6) {$\bs{a}_5$};
      \node[fill, circle, scale=0.5] at (4+2, 1) {};
      \node[anchor=east] at (4-0.1+2, 1) {$\bs{a}_6$};
      \node[anchor=north] at (2/3+4+.2+2, 3/3) {$\mathcal{T}$};
\end{tikzpicture}
\caption{\textit{A quadratic triangle $\mathcal{T}$ is obtained from the flat reference triangle $\widehat{T}$ via the map $F$ \cref{eq:mapping}. The $\bs{\hat{a}}_j$'s are given by $\bs{\hat{a}}_1=(0,0)$, $\bs{\hat{a}}_2=(1,0)$, and $\bs{\hat{a}}_3=(0,1)$, with midpoints $\bs{\hat{a}}_4$, $\bs{\hat{a}}_5$, and $\bs{\hat{a}}_6$; these verify $\varphi_i(\bs{\hat{a}}_j)=\delta_{ij}$. The $\bs{a}_j$'s on the quadratic triangle verify $\bs{a}_j=F(\bs{\hat{a}}_j)$. We display here a quadratic triangle $\mathcal{T}$ that is 2D but has a curved edge---quadratic triangles are, in general, 3D.}}
\label{fig:tri}
\end{figure}

\subsection{Locating the singularity (Step 2)}\label{sec:step2}

To compute $\bs{\hat{x}}_0$, the preimage of the singularity or near-singularity, we minimize $E(\bs{\hat{x}}) = \vert F(\bs{\hat{x}}) - \bs{x}_0\vert^2$, which we write as $E(\bs{\hat{x}}) = \vert e(\bs{\hat{x}})\vert^2$; for example,
\begin{align}
e(\bs{\hat{x}}) = \sum_{j=1}^6\varphi_j(\bs{\hat{x}})\bs{a}_j-\bs{x}_0
\end{align}
for quadratic triangles. Our optimization procedure employs the exact $2\times1$ gradient $\nabla E$,
\begin{align}
\nabla E(\bs{\hat{x}}) = 2
\begin{pmatrix}
\dsp e(\bs{\hat{x}})\cdot e_{\hat{x}_1}(\bs{\hat{x}}) \eqvspp
\dsp e(\bs{\hat{x}})\cdot e_{\hat{x}_2}(\bs{\hat{x}})
\end{pmatrix}.
\label{eq:gradient}
\end{align}
We also utilize the exact $2\times2$ Hessian matrix $H$,
\begin{align}
H(\bs{\hat{x}}) = 2
\begin{pmatrix}
\dsp e(\bs{\hat{x}})\cdot e_{\hat{x}_1\hat{x}_1}(\bs{\hat{x}}) + e_{\hat{x}_1}(\bs{\hat{x}})\cdot e_{\hat{x}_1}(\bs{\hat{x}}) &
\dsp e(\bs{\hat{x}})\cdot e_{\hat{x}_1\hat{x}_2}(\bs{\hat{x}}) + e_{\hat{x}_1}(\bs{\hat{x}})\cdot e_{\hat{x}_2}(\bs{\hat{x}}) \eqvspp
\text{SYMM.} & \dsp e(\bs{\hat{x}})\cdot e_{\hat{x}_2\hat{x}_2}(\bs{\hat{x}}) + e_{\hat{x}_2}(\bs{\hat{x}})\cdot e_{\hat{x}_2}(\bs{\hat{x}}) 
\end{pmatrix}.
\label{eq:hessian}
\end{align}
Note that the partial derivatives in \cref{eq:gradient} and \cref{eq:hessian} can be easily computed from the basis functions.

We start by initializing $\bs{\hat{x}}_0=(0,0)$. Our algorithm is then based on Newton's method,
\begin{align}
\bs{\hat{x}}_0^{\text{new}} = \bs{\hat{x}}_0 + \alpha p(\bs{\hat{x}}_0), \quad p(\bs{\hat{x}}_0) = -H(\bs{\hat{x}}_0)^{-1}\nabla E(\bs{\hat{x}}_0).
\label{eq:newton}
\end{align}
When the Hessian $H(\bs{\hat{x}}_0)$ is not positive definite, which we can check by evaluating its eigenvalues, we shift it by $\tau I_{2\times2}$ with $\tau = \max(0, \beta - \min(\lambda))$ and $\beta=10^{-3}$ \cite[Sec.~3.4]{nocedal2006}. The step length $\alpha$ is computed with a \textit{backtracking line search} with parameters $\rho=0.5$ and $c=10^{-4}$ \cite[Sec.~3.1]{nocedal2006}---starting from $\alpha=1$, it is progressively decreased via $\alpha:=\rho\alpha$ until it satisfies \textit{Armijo condition}
\begin{align}
E(\bs{\hat{x}}_0^{\text{new}}) \leq E(\bs{\hat{x}}_0) + c\alpha\nabla E(\bs{\hat{x}}_0)\cdot p(\bs{\hat{x}}_0).
\end{align}

Let us conclude this section by emphasizing that the output of the algorithm is a point $\bs{\hat{x}}_0\in\R^2$ such that its image $F(\bs{\hat{x}}_0)$ is the closest point to $\bs{x}_0$ on $\mathcal{S}\supset\mathcal{T}$. In certain cases, $\bs{\hat{x}}_0$ will be inside the reference element and hence $F(\bs{\hat{x}}_0)\in\mathcal{T}$, but in others, $\bs{\hat{x}}_0$ will be outside and $F(\bs{\hat{x}}_0)\in\mathcal{S}\setminus\overline{\mathcal{T}}$.

\subsection{Taylor expanding/subtracting (Step 3)}\label{sec:step3}

To first-order in $\delta\hat{x}=\vert\bs{\hat{x}} - \bs{\hat{x}}_0\vert$, we have
\begin{align}
F(\bs{\hat{x}}) - \bs{x}_0 =  J(\bs{\hat{x}}_0)(\bs{\hat{x}} - \bs{\hat{x}}_0) + h\bs{e}_h + \OO(\delta\hat{x}^2),
\end{align}
where the term $\OO(\delta\hat{x}^2)$, which is added componentwise, is bounded by
\begin{align*}
C\max_{i,j\in\left\{1,2\right\}}\max_{\bs{\hat{x}}\in\widehat{S}}\left\{\vert F_{{\hat{x}_i\hat{x}_j}}(\bs{\hat{x}})\vert\right\}\delta\hat{x}^2,
\end{align*}
for some constant $C>0$ and bounded domain $\widehat{S}\supset\widehat{T}$ that includes $\bs{\hat{x}}_0$. This yields
\begin{align}
\vert F(\bs{\hat{x}}) - \bs{x}_0\vert^2 = \vert J(\bs{\hat{x}}_0)(\bs{\hat{x}} - \bs{\hat{x}}_0)\vert^2 +  h^2 + h\OO(\delta\hat{x}^2) + \OO(\delta\hat{x}^3),
\end{align}
since $F(\bs{\hat{x}}_0)-\bs{x}_0$ is orthogonal to the tangent plane to $\mathcal{T}$ at $F(\bs{\hat{x}}_0)$ and $ J(\bs{\hat{x}}_0)(\bs{\hat{x}} - \bs{\hat{x}}_0)$ is in that~plane. Since the amplitudes of the derivatives of $F$ are of the order of $\rho$ (see \cref{eq:mapping}), the coefficients in the term with the Jacobian have size $\rho^2$ while the coefficient in $h\OO(\delta\hat{x}^2)$ has size $h\rho$; the latter may be be neglected since $h\ll\rho$. From the expansion of $R^2=\vert F(\bs{\hat{x}}) - \bs{x}_0\vert^2$, we retrieve that of $\psi R^{-1}$,
\begin{align}
\psi(\bs{\hat{x}})\vert F(\bs{\hat{x}})-\bs{x}_0\vert^{-1} \approx \psi(\bs{\hat{x}}_0)\left[\vert J(\bs{\hat{x}}_0)(\bs{\hat{x}} - \bs{\hat{x}}_0)\vert^2 + h^2\right]^{-\frac{1}{2}}.
\end{align}
This is how we derived the $\OO(\delta\hat{x}^{-1})$ term $T_{-1}$ in \cref{eq:Tn1}. Note that, when $h>0$, the above expression is a Taylor series but when $h=0$, $R^{-1}$ is indeed singular and this is merely an asymptotic expansion.

\subsection{Continuation approach (Step 4)}\label{sec:step4}

\paragraph{Singular case} A function $f:\widehat{T}\subset\R^2\rightarrow\R$ is said to be \textit{positive homogeneous of degree $r$} if there exists an integer $r$ such that $f(\lambda \bs{\hat{x}})=\lambda^r f(\bs{\hat{x}})$, for all $\bs{\hat{x}}\in\widehat{T}$ and $\lambda>0$. Such functions verify $\bs{\hat{x}}\cdot\nabla f(\bs{\hat{x}}) = rf(\bs{\hat{x}})$---\textit{Euler's homogeneous function theorem}. (To see this, differentiate with respect to $\lambda$ and evaluate at $\lambda=1$.) Using integration by parts and Euler's theorem, we get
\begin{align}
r\int_{\widehat{T}} f(\bs{\hat{x}})dS(\bs{\hat{x}}) = \int_{\partial\widehat{T}} f(\bs{\hat{x}})\bs{\hat{x}}\cdot\bs{\hat{\nu}}(\bs{\hat{x}})ds(\bs{\hat{x}}) - 2\int_{\widehat{T}} f(\bs{\hat{x}})dS(\bs{\hat{x}}),
\end{align}
with normal vector $\bs{\hat{\nu}}(\bs{\hat{x}})$. In the following, we will only deal with positive homogeneous functions of degree $r\geq-1$. For these functions, we can safely divide by $r+2$ to reach the following equation,
\begin{align}
I = \int_{\widehat{T}} f(\bs{\hat{x}})dS(\bs{\hat{x}}) = \frac{1}{r+2}\int_{\partial\widehat{T}}f(\bs{\hat{x}})\bs{\hat{x}}\cdot\bs{\hat{\nu}}(\bs{\hat{x}})ds(\bs{\hat{x}}) \quad\quad \text{($f$ homegeneous in $\bs{\hat{x}}$)}.
\label{eq:homogeneous}
\end{align}

\begin{figure}
\centering
\begin{tikzpicture}[scale = 0.80]
      \draw[thick] (2*0, 2*0) -- (2*0, 2*2) {};
      \draw[thick] (2*0, 2*0) -- (2*2, 2*0) {};
      \draw[thick] (2*0, 2*2) -- (2*2, 2*0) {};
      \node[fill, circle, scale=0.5] at (2*0, 2*0) {};
      \node[fill, circle, scale=0.5] at (2*2, 2*0) {};
      \node[fill, circle, scale=0.5] at (2*0, 2*2) {};
      \node[fill, circle, scale=0.5] at (2*2/3, 2*2/3) {};
      \node[anchor=west] at (2*2/3+0.1, 2*2/3-0.1) {$\bs{0}$};
      \node[anchor=west] at (0.7, 0.5) {$\hat{s}_1$};
      \node[anchor=west] at (1+0.1, 2) {$\hat{s}_2$};
      \node[anchor=west] at (0.25, 1.65) {$\hat{s}_3$};
      \draw[thick, dashed] (2*2/3, 2*2/3) -- (2, 2) {};
      \draw[thick, dashed] (2*2/3, 2*2/3) -- (0, 1+1/3) {};
      \draw[thick, dashed] (2*2/3, 2*2/3) -- (1+1/3, 0) {};
      \draw[->, thick, red] (0, 4)--(-1.41421356237, 4);      
      \draw[->, thick, red] (0, 0)--(0,-1.41421356237);      
  	  \draw[->, thick, red] (4, 0)--(5, 1);      
  	  \node[anchor=west, red] at (0+0.05,-1.41421356237) {$\bs{\hat{\nu}}_1$};
  	  \node[anchor=east, red] at (5-0.1, 1) {$\bs{\hat{\nu}}_2$};
  	  \node[anchor=north, red] at (-1.41421356237, 4-0.05) {$\bs{\hat{\nu}}_3$};
      \node[anchor=east] at (0-0.1, 0-0.1) {$\bs{\hat{a}}_1-\bs{\hat{x}}_0=(-\hat{x}_0,-\hat{y}_0)$};
      \node[anchor=west] at (4+0.1, 0-0.1) {$\bs{\hat{a}}_2-\bs{\hat{x}}_0=(1-\hat{x}_0,-\hat{y}_0)$};
      \node[anchor=west] at (0+0.1, 4+0.1) {$\bs{\hat{a}}_3-\bs{\hat{x}}_0=(-\hat{x}_0,1-\hat{y}_0)$};
      \node[anchor=north] at (2, 0-0.1) {$\bs{\hat{r}}_1(t)$};
      \node[anchor=west] at (2.4+0.15, 1.6+0.15) {$\bs{\hat{r}}_2(t)$};
      \node[anchor=east] at (0-0.1, 2) {$\bs{\hat{r}}_3(t)$};
\end{tikzpicture}
\caption{\textit{The triangle $\widehat{T}-\bs{\hat{x}}_0$ above is the reference triangle $\widehat{T}$ of \cref{fig:tri} shifted by $\bs{\hat{x}}_0=(\hat{x}_0,\hat{y}_0)$. The (signed) distances $\hat{s}_j$ to the edges are given by $\hat{s}_1=\hat{y}_0$, $\hat{s}_2=\sqrt{2}/2(1-\hat{x}_0-\hat{y}_0)$, and $\hat{s}_3=\hat{x}_0$.}}
\label{fig:shifted_tri}
\end{figure}
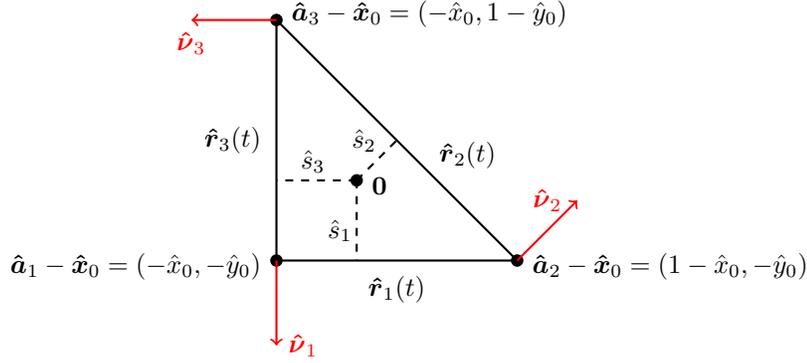

Consider, for example, the singular asymptotic term \cref{eq:Tn1} for $h=0$,
\begin{align}
T_{-1}(\bs{\hat{x}},0) = \frac{\psi(\bs{\hat{x}}_0)}{\vert J(\bs{\hat{x}}_0)(\bs{\hat{x}}-\bs{\hat{x}}_0)\vert}.
\end{align}
We translate by $\bs{\hat{x}}_0$ to make the integrand homogeneous and apply \cref{eq:homogeneous} with $r=-1$,
\begin{align}
I_{-1}(0) = \int_{\widehat{T}}T_{-1}(\bs{\hat{x}},0)dS(\bs{\hat{x}}) = \int_{\widehat{T}-\bs{\hat{x}}_0}\frac{\psi(\bs{\hat{x}}_0)}{\vert J(\bs{\hat{x}}_0)\bs{\hat{x}}\vert} dS(\bs{\hat{x}}) = \int_{\partial\widehat{T}-\bs{\hat{x}}_0}\frac{\psi(\bs{\hat{x}}_0)}{\vert J(\bs{\hat{x}}_0)\bs{\hat{x}}\vert}\bs{\hat{x}}\cdot\bs{\hat{\nu}}(\bs{\hat{x}})ds(\bs{\hat{x}}).
\end{align}
We observe that $\bs{\hat{x}}\cdot\bs{\hat{\nu}}(\bs{\hat{x}})$ is constant on each edge of $\widehat{T}-\bs{\hat{x}}_0$ and is equal to $\hat{s}_j=(\bs{\hat{a}}_j-\bs{\hat{x}}_0)\cdot\bs{\hat{\nu}}_j$, the (signed) distance between the origin and the edge with normal $\bs{\hat{\nu}}_j$; see \cref{fig:shifted_tri}. Hence, we get
\begin{align}
I_{-1}(0) =  \psi(\bs{\hat{x}}_0)\sum_{j=1}^3\hat{s}_j\int_{\partial\widehat{T}_j-\bs{\hat{x}}_0}\frac{ds(\bs{\hat{x}})}{\vert J(\bs{\hat{x}}_0)\bs{\hat{x}}\vert}.
\label{eq:In1_vtx_0}
\end{align}
We parametrize each edge,
\begin{align}
& \bs{\hat{r}}_1(t) = \left(-\hat{x}_0+\frac{t+1}{2}, -\hat{y}_0\right), \nonumber \\
& \bs{\hat{r}}_2(t) = \left(1-\hat{x}_0-\frac{t+1}{2}, -\hat{y}_0+\frac{t+1}{2}\right),  \label{eq:parametrization} \\
& \bs{\hat{r}}_3(t) =\left(-\hat{x}_0, 1-\hat{y}_0-\frac{t+1}{2}\right), \nonumber
\end{align}
with $-1\leq t\leq 1$, which leads to
\begin{align}
I_{-1}(0) = \psi(\bs{\hat{x}}_0)\sum_{j=1}^3\hat{s}_j\int_{-1}^1\frac{\vert\bs{\hat{r}}'_j(t)\vert}{\vert J(\bs{\hat{x}}_0)\bs{\hat{r}}_j(t)\vert}dt.
\label{eq:In1_param_0}
\end{align}

\paragraph{Near-singular case} There is a generalization of \cref{eq:homogeneous} and \cref{eq:In1_param_0} for functions $f(\cdot,h)$ that depend on a parameter $h\geq0$. Here, it is understood that $f(\cdot,h)$ is near-singular when $h>0$ and singular when $h=0$. Suppose that $f$ is positive homogeneous in both $\bs{x}$ and $h$, \textit{i.e.}, there exists an integer $r$ such that $f(\lambda\bs{\hat{x}},\lambda h) = \lambda^r f(\bs{\hat{x}},h)$, for all $\bs{\hat{x}}\in\widehat{T}$, $h\geq0$ and $\lambda>0$. Differentiating with respect to $\lambda$ and taking $\lambda=1$ yields
\begin{align}
\bs{\hat{x}}\cdot\nabla_{\bs{\hat{x}}}f(\bs{\hat{x}},h) + h\frac{\partial f}{\partial h}(\bs{\hat{x}},h) = rf(\bs{\hat{x}},h),
\label{eq:euler_h}
\end{align}
which is the analogue of Euler's homogeneous function theorem mentioned before. Let
\begin{align}
I(h) = \int_{\widehat{T}} f(\bs{\hat{x}},h)dS(\bs{\hat{x}}).
\end{align}
Using integration by parts together with \cref{eq:euler_h}, we can prove that
\begin{align}
hI'(h) - (r+2)I(h) = -\int_{\partial\widehat{T}} f(\bs{\hat{x}},h)\bs{\hat{x}}\cdot\bs{\hat{\nu}}(\bs{\hat{x}})ds(\bs{\hat{x}}).
\end{align}
Note that we recover \cref{eq:homogeneous} for $h=0$. For $h>0$, using variation of parameters, we get\footnote{We refer to \cite{rosen1995} for a careful derivation of these formulas.}
\begin{align}
I(h) = h^{r+2}\int_{\partial\widehat{T}}\bs{\hat{x}}\cdot\bs{\hat{\nu}}(\bs{\hat{x}})\int_h^{+\infty}\frac{f(\bs{\hat{x}},u)}{u^{r+3}}du\,ds(\bs{\hat{x}}) \quad\quad \text{($f$ homegeneous in $\bs{\hat{x}}$ and $h$)}.
\label{eq:homogeneous_h}
\end{align}
We observe, again, that  \cref{eq:homogeneous_h} approaches \cref{eq:homogeneous} as $h\to0$.

Take, for instance, the near-singular asymptotic term \cref{eq:Tn1} for $h>0$,
\begin{align}
T_{-1}(\bs{\hat{x}},h) = \frac{\psi(\bs{\hat{x}}_0)}{\sqrt{\vert J(\bs{\hat{x}}_0)(\bs{\hat{x}}-\bs{\hat{x}}_0)\vert^2 + h^2}}.
\end{align}
We translate by $\bs{\hat{x}}_0$ and apply \cref{eq:homogeneous_h} with $r=-1$,
\begin{align}
I_{-1}(h) = \int_{\widehat{T}}T_{-1}(\bs{\hat{x}},h)dS(\bs{\hat{x}}) = h\psi(\bs{\hat{x}}_0)\int_{\partial\widehat{T}-\bs{\hat{x}}_0}\bs{\hat{x}}\cdot\bs{\hat{\nu}}(\bs{\hat{x}})\int_h^{+\infty}\frac{du}{u^2\sqrt{\vert J(\bs{\hat{x}}_0)\bs{\hat{x}}\vert^2 + u^2}}ds(\bs{\hat{x}}).
\end{align}
Since, for any $h>0$ and nonzero vector $\bs{y}\in\R^3$,
\begin{align}
h\int_h^{+\infty}\frac{du}{u^2\sqrt{\vert\bs{y}\vert^2+u^2}} = \frac{\sqrt{\vert\bs{y}\vert^2+h^2} - h}{\vert\bs{y}\vert^2},
\end{align}
we arrive at the formula
\begin{align}
I_{-1}(h) = \psi(\bs{\hat{x}}_0)\sum_{j=1}^3\hat{s}_j\int_{\partial\widehat{T}_j-\bs{\hat{x}}_0}\frac{\sqrt{\vert J(\bs{\hat{x}}_0)\bs{\hat{x}}\vert^2+h^2}-h}{\vert J(\bs{\hat{x}}_0)\bs{\hat{x}}\vert^2}ds(\bs{\hat{x}}),
\label{eq:In1_vtx_h}
\end{align}
where the $\hat{s}_j$'s are those defined in \cref{fig:shifted_tri}. Using the parametrizations in \cref{eq:parametrization} yields
\begin{align}
I_{-1}(h) = \psi(\bs{\hat{x}}_0)\sum_{j=1}^3\hat{s}_j\int_{-1}^1\frac{\sqrt{\vert J(\bs{\hat{x}}_0)\bs{\hat{r}}_j(t)\vert^2+h^2} - h}{\vert J(\bs{\hat{x}}_0)\bs{\hat{r}}_j(t)\vert^2}\vert\bs{\hat{r}}'_j(t)\vert dt.
\label{eq:In1_param_h}
\end{align}
Note that the formula \cref{eq:In1_param_h} approaches \cref{eq:In1_param_0} as $h\to0$.

\subsection{Transplanted Gauss quadrature (Step 5)}\label{sec:step5}

Let us start this subsection by highlighting something crucial about the first integral in \cref{eq:In1_param_h}---similar remarks hold for the other two. As mentioned when introducing Step 5, if the origin is far from $\bs{\hat{r}}_1$ ($\hat{s}_1\gg0$ in \cref{fig:shifted_tri}), then the integrand is analytic and $n$-point Gauss quadrature converges exponentially fast~\cite[Thm.~19.3]{trefethen2019}; the quadrature error decreases like $\OO(\rho^{-2n})$, and $\hat{s}_1\gg0$ yields $\rho\gg1$. However, if the origin is close to $\bs{\hat{r}}_1$ ($\hat{s}_1\approx0$), the integrand is analytic but nearly singular, convergence is exponential but~$\rho\approx1$. To circumvent this issue, we use transplanted quadrature rules obtained from conformal maps \cite{hale2008}. 

Suppose $f$ is an analytic function in some Bernstein ellipse $E_\rho$ for some $\rho>1$,\footnote{The ellipse $E_\rho$ is the set of points $(x,y)$ such that $x^2/a_\rho^2+y^2/b_\rho^2=1$ with $a_\rho=1/2(\rho+\rho^{-1})$ and $b_\rho=1/2(\rho-\rho^{-1})$.} and that it has complex singularities near $[-1,1]$. We seek to calculate
\begin{align}
I(f) = \int_{-1}^1f(t)dt.
\label{eq:int}
\end{align}
In this case, $n$-point \textit{Gauss quadrature},
\begin{align}
I(f) \approx I_n(f) = \sum_{k=1}^nw_kf(t_k)
\label{eq:gauss_quad}
\end{align}
with Gauss weights $\{w_k\}$ and nodes $\{t_k\}$, fails to converge rapidly. The idea behind transplanted quadrature goes like this. Let $\Omega_\rho$ be an open set in the complex plane containing $[-1,1]$ inside of which the function $f$ is analytic and let $g$ be an analytic function in $E_\rho$ satisfying
\begin{align}
g(E_\rho)\subset\Omega_\rho, \quad g(-1) = -1, \quad g(1) = 1.
\end{align}
By Cauchy's theorem for analytic functions, the integral of $f$ over $g([-1,1])$---an analytic curve in the complex plane---is the same as the integral of $f$ over $[-1,1]$, thus we have
\begin{align}
I(f) = \int_{-1}^1g'(t)f(g(t))dt.
\label{eq:trans_int}
\end{align}
We can approximate the above integral by the following \textit{transplanted Gauss quadrature rule},
\begin{align}
I(f) \approx I_n(f) = \sum_{k=1}^nw_kg'(t_k)f(g(t_k)).
\label{eq:trans_quad}
\end{align}

In the following, we shall consider functions $f_{\mu,\nu}$ of the form
\begin{align}
f_{\mu,\nu}(t) = \frac{1}{\sqrt{(t - \mu)^2 + \nu^2}}, \quad \mu\in\R, \quad \nu>0, \quad t\in[-1,1],
\label{eq:f_munu}
\end{align}
with complex singularities at $t=\mu\pm i\nu$ and exact integral $I(\mu,\nu)=a_{\mu,\nu} + b_{\mu,\nu}$ with
\begin{align}
a_{\mu,\nu} = \mrm{arcsinh}\left(\frac{1-\mu}{\nu}\right), \quad b_{\mu,\nu} = \mrm{arcsinh}\left(\frac{1+\mu}{\nu}\right).
\end{align}
We will utilize the map from \cite{tee2006} defined by
\begin{align}
g_{\mu,\nu}(z) = \mu + \nu\sinh\left[(a_{\mu,\nu} + b_{\mu,\nu})\frac{z-1}{2} + a_{\mu,\nu}\right].
\label{eq:g_munu}
\end{align}
For this particular real-valued function $g$, which verifies $g([-1,1])=[-1,1]$ and $g'>0$ on $[-1,1]$, the transplanted integral \cref{eq:trans_int} can be retrieved via the substitution $u=g(t)$, and the transplanted formula \cref{eq:trans_quad} is simply Gauss quadrature applied to this change of variables.

Functions of the form of \cref{eq:f_munu} are particularly relevant for us. For example, if $h=0$, $\bs{\hat{x}}_0=(0,\epsilon)$ and $J(\bs{\hat{x}}_0)=I_{3\times2}$, the integral on $\bs{\hat{r}}_1$ in \cref{eq:In1_param_h} reads 
\begin{align}
\int_{-1}^1\frac{\vert \bs{\hat{r}}_1'(t)\vert}{\vert\bs{\hat{r}}_1(t)\vert}dt = \int_{-1}^1\frac{dt}{\sqrt{(t+1)^2 + 4\epsilon^2}} = I(-1,2\epsilon),
\label{eq:r1_integrand1}
\end{align}
with complex singularities at $t= -1\pm2i\epsilon$. Since the largest Bernstein ellipse inside which it can be analytically continued has a parameter $\rho\approx1 + \sqrt{2\epsilon}$, $n$-point Gauss quadrature converges exponentially at a slow rate $(1+\sqrt{2\epsilon})^{-2n}$---when $\epsilon$ is small, the error is almost constant. Transplanted Gauss quadrature with $g_{-1,2\epsilon}$, on the other hand, is exact; see \cref{fig:quad_conv1}. Note that, for $\bs{\hat{x}}_0 = (\hat{x}_0,\epsilon)$ instead of $(0,\epsilon)$, the integral on $\bs{\hat{r}}_1$ has singularities at $t=-1+2\hat{x}_0 \pm 2i\epsilon$ and equals $I(-1+2\hat{x}_0,2\epsilon)$.

\begin{figure}
\centering
\includegraphics[scale=0.5]{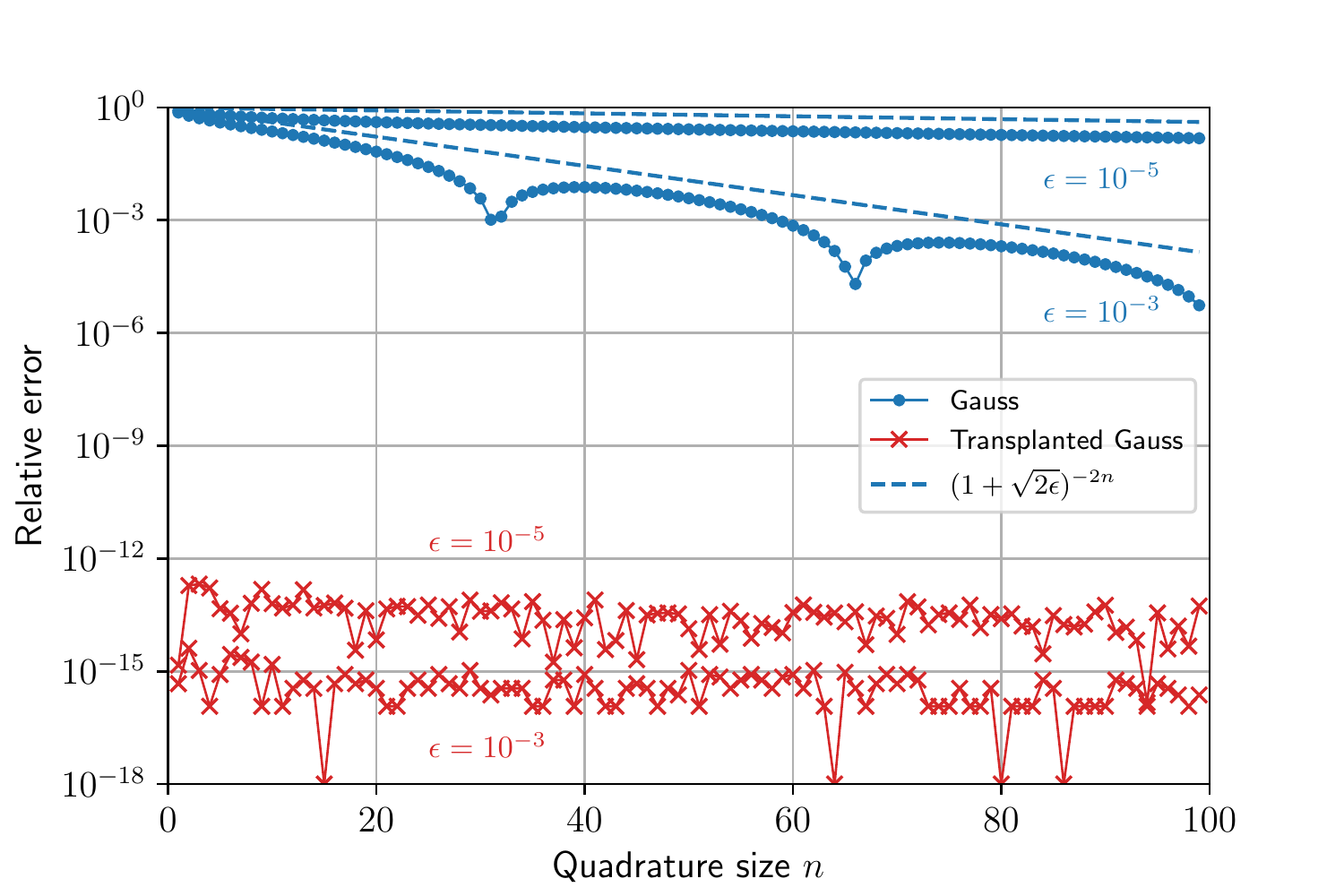}
\caption{\textit{Gauss quadrature \cref{eq:gauss_quad} converges exponentially for the function $t\mapsto1/\sqrt{(t+1)^2+4\epsilon^2}$, but at a slow rate $(1+\sqrt{2\epsilon})^{-2n}$. As $\epsilon\to0$, the error ``plateaus'' at $\OO(1)$ values. Transplanted Gauss quadrature \cref{eq:trans_quad} with the map $g_{-1,2\epsilon}$ defined in \cref{eq:g_munu} gives machine precision from $n=1$.}}
\label{fig:quad_conv1}
\end{figure}

Let us emphasize that, in general, the integrals in \cref{eq:In1_param_h} will not be as simple as \cref{eq:r1_integrand1}---the term under the square root will be a more complicated polynomial in $t$ and $\epsilon$ because of the Jacobian. We could compute the exact singularities $\mu\pm i\nu$, construct the corresponding $g_{\mu,\nu}$ and would observe full accuracy from $n=1$ for transplanted Gauss quadrature. However, for simplicity, we use the same transplanted rule with $\mu=-1+2\hat{x}_0$ and $\nu=2\epsilon$ for all integrands with $1/\vert J(\bs{\hat{x}}_0)\bs{\hat{r}}_1(t)\vert$ near-singularity, independently of the value of $J(\bs{\hat{x}}_0)$. By doing so, we only misplace the singularities by a factor $\OO(\epsilon)$.\footnote{For example, for a Jacobian $J$ with columns $J_1=(a,c,0)$ and $J_2=(b,d,0)$, the singularity is at 
\begin{align}
t=-1+2\hat{x}_0+2\epsilon\frac{ab+cd}{a^2+c^2} \pm 2i\epsilon\frac{ad-bc}{a^2+c^2}.
\end{align}} (Similarly, on $\bs{\hat{r}}_2$ and $\bs{\hat{r}}_3$, we choose $\mu=-1+2\hat{y}_0$ and $\nu=2\hat{s}_2$, and $\mu=-1+2(1-\hat{y}_0)$ and $\nu=2\hat{s}_3$ for all integrands.) To illustrate this, we use transplanted Gauss quadrature with $g_{0,2\epsilon}$ for integrating $f_{0,5\epsilon}$ in \cref{fig:quad_conv2}---transplanted Gauss quadrature is no longer exact but converges much faster than Gauss quadrature. We prove these observations in the following theorem.

\begin{figure}
\centering
\includegraphics[scale=0.5]{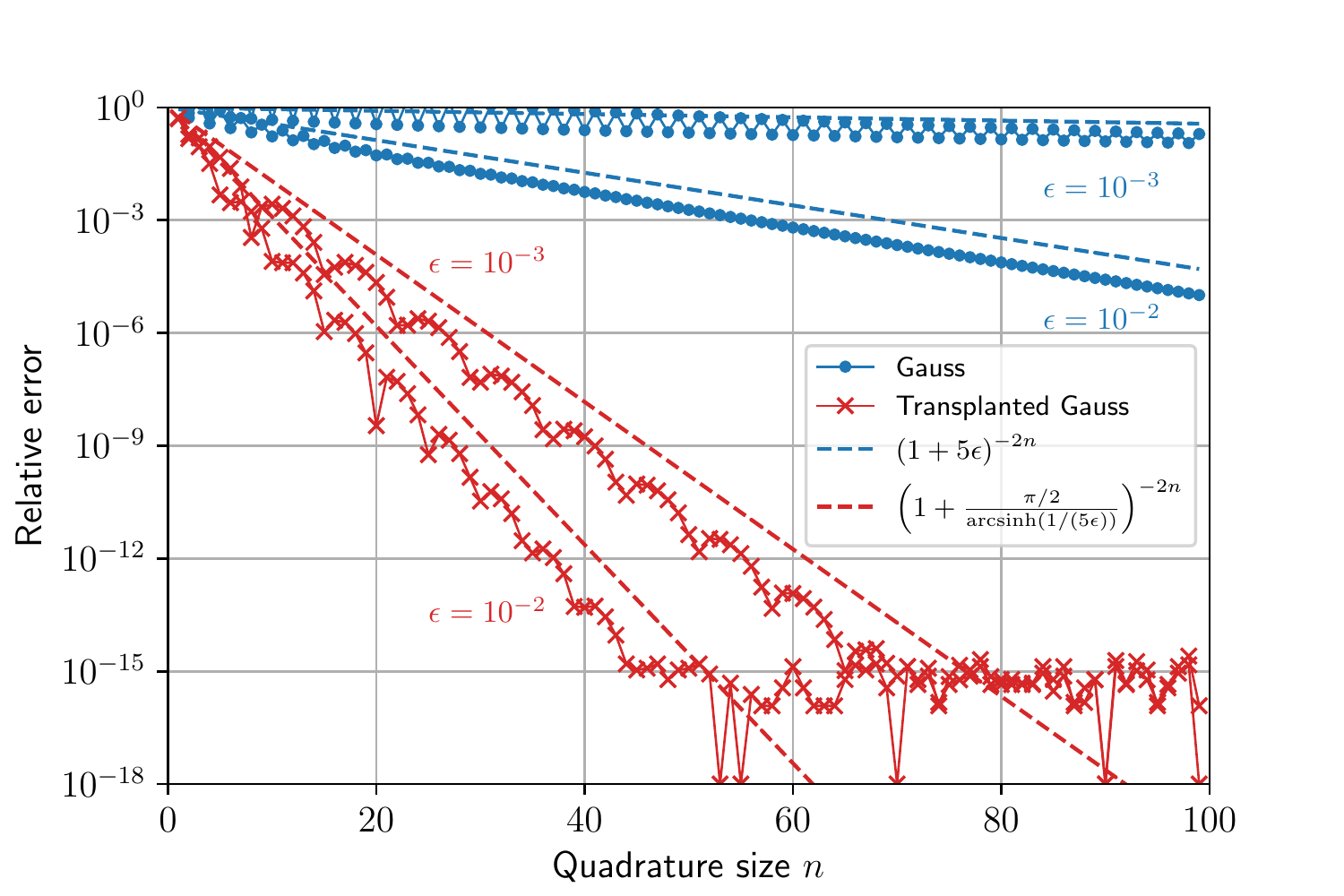}
\caption{\textit{Gauss quadrature \cref{eq:gauss_quad} converges exponentially for the function $t\mapsto1/\sqrt{t^2+25\epsilon^2}$, but at a slow rate $(1+5\epsilon)^{-2n}$. Transplanted Gauss quadrature \cref{eq:trans_quad} with the map $g_{0,2\epsilon}$ defined in \cref{eq:g_munu} is no longer exact but converges exponentially at a faster rate, described in \cref{thm:convergence}.}}
\label{fig:quad_conv2}
\end{figure}

\begin{theorem}\label{thm:convergence} 
Let $-1\leq\mu\leq1$, $0<\nu<1$, $f_{\mu,\nu}$ and $g_{\mu,\nu}$ as in \cref{eq:f_munu}--\cref{eq:g_munu}, and
\begin{align}
C_\nu = \frac{\pi/2}{\mrm{arcsinh}(2/\nu)}.
\end{align}
Let $\rho_0(x) = x + \sqrt{1 + x^2}$ and $\rho_1(x) = \rho_0(x) + \sqrt{2x\rho_0(x)}$. In the following statements, by super-exponential we mean that the $n$-point quadrature error $\vert I(\mu,\nu) - I_n(\mu,\nu)\vert$ decreases like $\OO(\rho^{-2n})$ for all $\rho>1$, and by exponential we mean $\OO(\rho^{-2n})$ for some $\rho>1$.

$(1)$ Transplanted Gauss quadrature \cref{eq:trans_quad} with $g_{\mu,\nu}$ applied to the product $f_{\mu,\nu}^k\times\psi$ converges exponentially for any entire function $\psi$ and nonzero integer power $k$ with $\rho=\rho_0(C_{2\nu})$ if $\mu=0$, $\rho=\rho_1(C_\nu)$ if $\mu=\pm1$, and $\rho_0(C_{2\nu})<\rho<\rho_1(C_\nu)$ if $-1<\mu\neq0<1$. This gives, as $\nu\to0$,
\begin{align}
& \rho \approx 1 + \frac{\pi/2}{\mrm{arcsinh}(1/\nu)}, && \text{if $\mu=0$}, \nonumber \\
& \rho \approx 1 + \sqrt{\frac{\pi}{\mrm{arcsinh}(2/\nu)}}, && \text{if $\mu=\pm1$}, \\
& 1 + \frac{\pi/2}{\mrm{arcsinh}(1/\nu)} \lesssim \, \rho \, \lesssim 1 + \sqrt{\frac{\pi}{\mrm{arcsinh}(2/\nu)}}, && \text{if $-1<\mu\neq0<1$}. \nonumber
\end{align}
When $k\leq1$, the convergence improves to being super-exponential, and when $k=1$ and $\psi$ is a constant, the quadrature is exact with a single quadrature point/weight.

$(2)$ Gauss quadrature \cref{eq:gauss_quad} applied to any of the functions in $(1)$ converges exponentially with $\rho=\rho_0(\nu)$ if $\mu=0$, $\rho=\rho_1(\nu/2)$ if $\mu=\pm1$, and $\rho_0(\nu)<\rho<\rho_1(\nu/2)$ if $-1<\mu\neq0<1$. This gives the following approximations for $\rho$ as $\nu\to0$,
\begin{align}
& \rho \approx 1+\nu, && \text{if $\mu=0$}, \nonumber \eqvsp
& \rho \approx 1+\sqrt{\nu}, && \text{if $\mu=\pm1$}, \eqvsp
& 1+\nu \lesssim \, \rho \, \lesssim 1+\sqrt{\nu}, && \text{if $-1<\mu\neq0<1$}. \nonumber
\end{align}

$(3)$ Let $\mu'=\mu+\dmu\nu$ and $\nu'=\dnu\nu$ for some $\dmu\in\R$ and $\dnu>0$, and let
\begin{align}
0\leq I = \frac{\mrm{Im}\left(\mrm{arcsinh}\left(\dmu + i\dnu\right)\right)}{\pi/2}\leq1.
\end{align}
Transplanted Gauss quadrature \cref{eq:trans_quad} with $g_{\mu,\nu}$ applied to the product $f_{\mu',\nu'}^k\times\psi$ converges exponentially for any entire function $\psi$ and nonzero integer power $k$ with $\rho\geq\rho_0(2IC_{\nu})$ if $\mu=\pm1$ and $\rho\geq\rho_0(IC_{2\nu})$ if $\vert\mu\vert<1$. This gives the following approximations for $\rho$ as $\nu\to0$,
\begin{align}
& \rho \gtrsim 1 + \frac{I\pi}{\mrm{arcsinh}(2/\nu)}, && \text{if $\mu=\pm1$}, \nonumber \\
& \rho \gtrsim 1 + \frac{I\pi/2}{\mrm{arcsinh}(1/\nu)}, && \text{if $\vert\mu\vert<1$}.
\end{align}
\end{theorem}

\begin{proof}
(1) Transplanted quadrature with $g_{\mu,\nu}$ yields integrands of the form
\begin{align}
g'_{\mu,\nu}(t)f^k_{\mu,\nu}(g_{\mu,\nu}(t))\psi(g_{\mu,\nu}(t)) = \frac{a_{\mu,\nu}+b_{\mu,\nu}}{2}\cosh^{1-k}\left[(a_{\mu,\nu}+b_{\mu,\nu})\frac{t-1}{2}+a_{\mu,\nu}\right]\psi(g_{\mu,\nu}(t)).
\label{eq:trans_integrand}
\end{align}
The closest singularities occur when the input of the hyperbolic cosine equals $\pm i\pi/2$, that is,
\begin{align}
t = -1 + \frac{2b_{\mu,\nu}}{a_{\mu,\nu}+b_{\mu,\nu}} \pm i\frac{\pi}{a_{\mu,\nu}+b_{\mu,\nu}}.
\label{eq:trans_singularities}
\end{align}
If $\mu=0$, then $a_{\mu,\nu} = b_{\mu,\nu} = \mrm{arcsinh}(1/\nu)$ and $t = \pm iC_{2\nu}$. In this case, the largest Bernstein ellipse $E_{\rho}$ verifies $\rho = \rho_0(C_{2\nu})$. If $\mu=-1$, then $b_{\mu,\nu}=0$, $a_{\mu,\nu}=\mrm{arcsinh}(2/\nu)$ and $t = -1 \pm 2iC_\nu$. Here, the largest ellipse $E_{\rho}$ verifies $\rho = \rho_1(C_\nu)$. Similar calculations can be carried out for $\mu=1$. For the other $\mu$'s, the singularities lie outside $E_{\rho_0(C_{2\nu})}$ but inside $E_{\rho_1(C_\nu)}$, and therefore $\rho_0(C_{2\nu})<\rho<\rho_1(C_\nu)$. (The approximations as $\nu\to0$ are straightforward asymptotic expansions.) When $k\leq1$, the function in \cref{eq:trans_integrand} is a product-composition of entire functions---convergence is super-exponential. Finally, when $k=1$ and $\psi$ is a constant, \cref{eq:trans_integrand} is a constant---the quadrature is exact with $n=1$.

(2) The singularities are at $t=\mu\pm i\nu$. The formulas can be derived with the same techniques.

(3) Singularities occur when the denominator of the function $f_{\mu',\nu'}\circ g_{\mu,\nu}$ vanishes, \textit{i.e.},
\begin{align}
\left(X(t) - \dmu\right)^2 + \dnu^2 = 0, 
\quad X(t) = \sinh\left[(a_{\mu,\nu} + b_{\mu,\nu})\frac{t-1}{2} + a_{\mu,\nu}\right].
\end{align}
This leads to $X(t) = \dmu\pm i\dnu$ and a couple of candidate pairs for the closest singularities,
\begin{align}
t = -1 + \frac{2b_{\mu,\nu} + \pi R}{a_{\mu,\nu} + b_{\mu,\nu}} \pm i\frac{\pi I}{a_{\mu,\nu} + b_{\mu,\nu}}, \quad 
t = -1 + \frac{2b_{\mu,\nu} - \pi R}{a_{\mu,\nu} + b_{\mu,\nu}} \pm i\frac{\pi(2 - I)}{a_{\mu,\nu} + b_{\mu,\nu}},
\end{align}
with
\begin{align}
R = \frac{\mrm{Re}\left(\mrm{arcsinh}\left(\dmu + i\dnu\right)\right)}{\pi/2}, \quad I = \frac{\mrm{Im}\left(\mrm{arcsinh}\left(\dmu + i\dnu\right)\right)}{\pi/2}.
\end{align}
If $\mu=-1$, either $t=-1 + 2RC_\nu \pm 2iIC_\nu$ or $t=-1 - 2RC_\nu \pm 2i(2-I)C_\nu$ may be the closest singularities. These both lie outside the ellipse associated with purely imaginary singularities, for which $\rho=\rho_0(2IC_\nu)$, and, therefore, $\rho\geq\rho_0(2IC_\nu)$ (crude estimate). Similar calculations for $\mu=1$. If $\mu=0$, then $t = RC_{2\nu} \pm iIC_{2\nu}$ are the closest singularities since $2-I\geq I$. These lie outside the ellipse corresponding to purely imaginary singularities, for which $\rho=\rho_0(IC_{2\nu})$; hence, we arrive at $\rho\geq\rho_0(IC_{2\nu})$. For the other $\mu$'s, the singularities lie outside the ellipse corresponding to $\mu=0$. 
\end{proof}

Let us conclude this section by noting that the integral on the edge $\bs{\hat{r}}_1$ in \cref{eq:In1_param_h} does not need to be computed when the origin is on $\bs{\hat{r}}_1$, that is, when $\epsilon=0$, since it is multiplied by $\hat{s}_1=\epsilon$ and
\begin{align}
\epsilon\,I(-1,2\epsilon) = \epsilon\,\mrm{arcsinh}\left(\frac{1}{\epsilon}\right) = \epsilon\log\left(\frac{2}{\epsilon}\right) + \frac{1}{4}\epsilon^3 + \ldots \underset{\epsilon\to0}{\longrightarrow} 0.
\end{align}
The latter observation also shows that the contribution $\epsilon I(-1,2\epsilon)$ on $\bs{\hat{r}}_1$ becomes negligible with respect to those on $\bs{\hat{r}}_2$ and $\bs{\hat{r}}_3$ as $\epsilon\to0$. Therefore, if we were to employ Gauss quadrature, $I(-1,2\epsilon)$ would be computed with an error $\OO(1)$ but this would only introduce an error $\OO(\epsilon\log(2/\epsilon))$ in \cref{eq:In1_param_h}. 

\section{Extension to higher-order approximations}\label{sec:high-order}

We present in this section our method based on high-order Taylor series. The first two steps are the same---we go directly to Step 3.

\paragraph{Step 3. Taylor expanding/subtracting} Let $\delta\bs{\hat{x}} = \bs{\hat{x}} - \bs{\hat{x}}_0 = (\delta\hat{x}_1,\delta\hat{x}_2)$, $ J_0 = J(\bs{\hat{x}}_0)$, and
\begin{align}
& \psi_0 = \psi(\bs{\hat{x}}_0) = \OO(1), \nonumber \eqvspp
& \psi'_0(\delta\bs{\hat{x}}) = \psi_{\hat{x}_1}(\bs{\hat{x}}_0)\delta\hat{x}_1 + \psi_{\hat{x}_2}(\bs{\hat{x}}_0)\delta\hat{x}_2 = \OO(\delta\hat{x}), \eqvspp
& \psi''_0(\delta\bs{\hat{x}}) = \frac{1}{2}\psi_{\hat{x}_1\hat{x}_1}(\bs{\hat{x}}_0)\delta\hat{x}_1^2 + \frac{1}{2}\psi_{\hat{x}_2\hat{x}_2}(\bs{\hat{x}}_0)\delta\hat{x}_2^2 + \psi_{\hat{x}_1\hat{x}_2}(\bs{\hat{x}}_0)\delta\hat{x}_1\delta\hat{x}_2 = \OO(\delta\hat{x}^2). \nonumber
\end{align}
Besides the $\OO(\delta{\hat{x}}^{-1})$ term $T_{-1}$, we also compute the $\OO(1)$ and $\OO(\delta\hat{x})$ terms $T_0$ and $T_1$,
\begin{align}
T_0(\bs{\hat{x}},h) & = \frac{\psi'_0(\delta\bs{\hat{x}})}{\left[\vert J_0\delta\bs{\hat{x}}\vert^2 + h^2\right]^{\frac{1}{2}}}
- \frac{h\psi_0}{2}\sum_{i=1}^3a_i\frac{\delta\hat{x}_1^{3-i}\delta\hat{x}_2^{i-1}}{\left[\vert J_0\delta\bs{\hat{x}}\vert^2 + h^2\right]^{\frac{3}{2}}}
- \frac{\psi_0}{2}\sum_{i=1}^4c_i\frac{\delta\hat{x}_1^{4-i}\delta\hat{x}_2^{i-1}}{\left[\vert J_0\delta\bs{\hat{x}}\vert^2 + h^2\right]^{\frac{3}{2}}}, \nonumber \eqvspp
T_1(\bs{\hat{x}},h) & = \frac{\psi''_0(\delta\bs{\hat{x}})}{\left[\vert J_0\delta\bs{\hat{x}}\vert^2 + h^2\right]^{\frac{1}{2}}}
+ h\sum_{i=1}^4e_i\frac{\delta\hat{x}_1^{4-i}\delta\hat{x}_2^{i-1}}{\left[\vert J_0\delta\bs{\hat{x}}\vert^2 + h^2\right]^{\frac{3}{2}}} 
+ h^2\sum_{i=1}^5f_i\frac{\delta\hat{x}_1^{5-i}\delta\hat{x}_2^{i-1}}{\left[\vert J_0\delta\bs{\hat{x}}\vert^2 + h^2\right]^{\frac{5}{2}}} \label{eq:T} \eqvspp
& + \sum_{i=1}^5g_i\frac{\delta\hat{x}_1^{5-i}\delta\hat{x}_2^{i-1}}{\left[\vert J_0\delta\bs{\hat{x}}\vert^2 + h^2\right]^{\frac{3}{2}}}
+ \sum_{i=1}^7h_i\frac{\delta\hat{x}_1^{7-i}\delta\hat{x}_2^{i-1}}{\left[\vert J_0\delta\bs{\hat{x}}\vert^2 + h^2\right]^{\frac{5}{2}}}, \nonumber
\end{align}
and add them to/subtract them from \cref{eq:step1} to regularize the 2D integrals,
\begin{align}
I(\bs{x}_0) & = \int_{\widehat{T}}T_{-1}(\bs{\hat{x}},h)dS(\bs{\hat{x}}) + \int_{\widehat{T}}T_0(\bs{\hat{x}},h)dS(\bs{\hat{x}}) + \int_{\widehat{T}}T_1(\bs{\hat{x}},h)dS(\bs{\hat{x}}) \nonumber \\
& + \int_{\widehat{T}}\left[\frac{\psi(\bs{\hat{x}})}{\vert F(\bs{\hat{x}}) - \bs{x}_0\vert} - T_{-1}(\bs{\hat{x}},h) - T_0(\bs{\hat{x}},h) - T_1(\bs{\hat{x}},h)\right]dS(\bs{\hat{x}}).
\end{align}
Check out \cref{sec:taylor_R2} and \cref{sec:taylor_psiRn1} for the coefficients in \cref{eq:T}.

\paragraph{Steps 4--5. Continuation approach and transplanted Gauss quadrature} Let
\begin{align}
I_0(h) = \int_{\widehat{T}}T_0(\bs{\hat{x}},h)dS(\bs{\hat{x}}), \quad
I_1(h) = \int_{\widehat{T}}T_1(\bs{\hat{x}},h)dS(\bs{\hat{x}}).
\label{eq:I_tri}
\end{align}
Using the continuation approach, we obtain 1D integrals along the boundary $\partial\widehat{T}$ of $\widehat{T}$, which are, again, analytic but nearly singular; see \cref{sec:I0_I1}. The story is the same---Gauss quadrature will be inefficient, while transplanted Gauss quadrature will perform well, as guaranteed by \cref{thm:convergence}.

\section{Numerical examples}\label{sec:numerics}

\paragraph{Singular/near-singular 2D integrals over a quadratic triangle} Consider the triangle $\mathcal{T}$ given~by
\begin{align}
& \bs{a}_1 = \bs{\hat{a}}_1 = (0,0,0), \quad\quad \bs{a}_4 = \bs{\hat{a}}_4 = (1/2,0,0), \nonumber \\
& \bs{a}_2 = \bs{\hat{a}}_2 = (1,0,0), \quad\quad \bs{a}_5 = (a,b,c) \neq \bs{\hat{a}}_5=(1/2,1/2,0), \label{eq:exp_tri}\\
& \bs{a}_3 = \bs{\hat{a}}_3 = (0,1,0), \quad\quad \bs{a}_6 = \bs{\hat{a}}_6 = (0,1/2,0), \nonumber
\end{align}
which is displayed in \cref{fig:exp_tri}. The mapping $F:\widehat{T}\mapsto\mathcal{T}$ and its Jacobian matrix $J$ are given by
\begin{align}
F(\bs{\hat{x}}) = \begin{pmatrix}
\hat{x}_1 + 2(2a-1)\hat{x}_1\hat{x}_2 \eqvsp
\hat{x}_2 + 2(2b-1)\hat{x}_1\hat{x}_2 \eqvsp
4c\hat{x}_1\hat{x}_2
\end{pmatrix}, \quad
J(\bs{\hat{x}}) = \begin{pmatrix}
1 + 2(2a-1)\hat{x}_2 & 2(2a-1)\hat{x}_1 \eqvsp
2(2b-1)\hat{x}_2 & 1 + 2(2b-1)\hat{x}_1 \eqvsp
4c\hat{x}_2 & 4c\hat{x}_1
\end{pmatrix}.
\end{align}
We take $a=0.6$, $b=0.7$ and $c=0.5$, and compute the following integrals for different values of $\bs{x}_0$,
\begin{align}
I(\bs{x}_0) = \int_\mathcal{T}\frac{dS(\bs{x})}{\vert\bs{x}-\bs{x}_0\vert}.
\label{eq:Ix0}
\end{align}

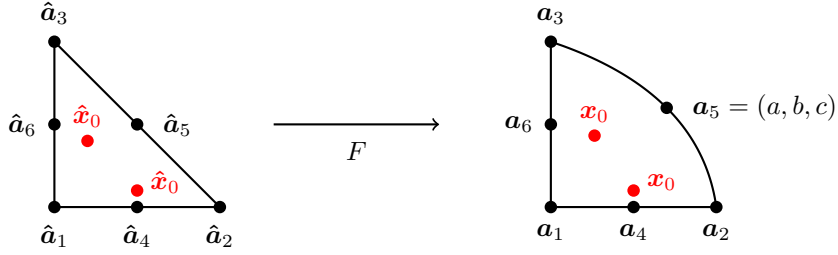
\begin{figure}
\centering
\begin{tikzpicture}[scale = 1.10]
      \draw[thick] (0, 0) -- (2, 0) {};
      \draw[thick] (0, 0) -- (0, 2) {};
      \draw[thick] (2, 0) -- (0, 2) {};
      \node[fill, circle, scale=0.5] at (0, 0) {};
      \node[anchor=north] at (0, 0-0.1) {$\bs{\hat{a}}_1$};
      \node[fill, circle, scale=0.5] at (2, 0) {};
      \node[anchor=north] at (2, 0-0.1) {$\bs{\hat{a}}_2$};
      \node[fill, circle, scale=0.5] at (0, 2) {};
      \node[anchor=south] at (0, 2+0.1) {$\bs{\hat{a}}_3$};
      \node[fill, circle, scale=0.5] at (1, 0) {};
      \node[anchor=north] at (1, 0-0.1) {$\bs{\hat{a}}_4$};
      \node[fill, circle, scale=0.5] at (1, 1) {};
      \node[anchor=west] at (1+0.2, 1) {$\bs{\hat{a}}_5$};
      \node[fill, circle, scale=0.5] at (0, 1) {};
      \node[anchor=east] at (0-0.1, 1) {$\bs{\hat{a}}_6$};
      \node[fill, circle, scale=0.5, color=red] at (0.4, 0.8) {};
      \node[anchor=south, color=red] at (0.4, 0.8+0.05) {$\bs{\hat{x}}_0$};
      \node[fill, circle, scale=0.5, color=red] at (1, 0.2) {};
      \node[anchor=west, color=red] at (1+0.05, 0.2+0.1) {$\bs{\hat{x}}_0$};      
      \draw[thick, ->] (1.65+1, 1) -- (3.65+1, 1) {};
      \node[anchor=north] at (2.65+1, 1-0.1) {$F$};
      \draw[thick] (4+2, 0) -- (6+2, 0) {};
      \draw[thick] (4+2, 0) -- (4+2, 2) {};
      \draw[thick, domain=-1:1, smooth, variable=\theta, black] plot ({6+2*(1-(\theta+1)/2) + 2*2*(2*0.7-1)*(1-(\theta+1)/2)*(\theta+1)/2}, {2*(\theta+1)/2 + 2*2*(2*0.6-1)*(1-(\theta+1)/2)*(\theta+1)/2});
      \node[fill, circle, scale=0.5] at (4+2, 0) {};
      \node[anchor=north] at (4+2, 0-0.1) {$\bs{a}_1$};
      \node[fill, circle, scale=0.5] at (6+2, 0) {};
      \node[anchor=north] at (6+2, 0-0.1) {$\bs{a}_2$};
      \node[fill, circle, scale=0.5] at (4+2, 2) {};
      \node[anchor=south] at (4+2, 2+0.1) {$\bs{a}_3$};
      \node[fill, circle, scale=0.5] at (5+2, 0) {};
      \node[anchor=north] at (5+2, 0-0.1) {$\bs{a}_4$};
      \node[fill, circle, scale=0.5] at (6+2*0.7, 2*0.6) {};
      \node[anchor=west] at (6+2*0.7+0.2, 2*0.6) {$\bs{a}_5=(a,b,c)$};
      \node[fill, circle, scale=0.5] at (4+2, 1) {};
      \node[anchor=east] at (4-0.1+2, 1) {$\bs{a}_6$};
      \node[fill, circle, scale=0.5, color=red] at (6+2*0.264, 2*0.432) {};
      \node[anchor=south, color=red] at (6+2*0.264, 2*0.432+0.05) {$\bs{x}_0$};
      \node[fill, circle, scale=0.5, color=red] at (1+6, 0.2) {};
      \node[anchor=west, color=red] at (1+0.05+6, 0.2+0.1) {$\bs{x}_0$};      
\end{tikzpicture}
\caption{\textit{The quadratic triangle we consider is on the right---we choose $c=0.5$ to get a 3D triangle. We take points $\bs{\hat{x}}_0=(0.2,0.4)$ (near the center) and $\bs{\hat{x}}_0=(0.5,\epsilon)$ for small $\epsilon$ (near the lower edge).}}
\label{fig:exp_tri}
\end{figure}

We first take points $\bs{x}_0=F(0.2,0.4)$ (singular) and $\bs{x}_0=F(0.2,0.4)+10^{-4}\bs{z}$ (near-singular). These are points near the center of $\mathcal{T}$; the first one is on $\mathcal{T}$ and the second one is slightly above it. We compute the integrals \cref{eq:Ix0} in Mathematica to $14$-digit accuracy and compare them with the values obtained with our method using $T_{-1}$, $T_0$, and $T_1$ regularization. We use $N=n^2$ quadrature points for the 2D integrals and $10n$ points for the 1D integrals, with $2\leq n\leq200$. The 2D integrals are computed with Gauss quadrature on triangles \cite{lether1976}, while the 1D integrals are computed with 1D Gauss quadrature ($\bs{\hat{x}}_0$ is far from the edges). We plot the relative errors in \cref{fig:exp1}. The convergence is linear with the number of quadrature points when using the method of \cref{sec:first-order}, and improves to quadratic when using that of \cref{sec:high-order}. (Note that the errors come from the 2D integrals---the 1D integrals are analytic so their errors are much smaller in comparison.)

\begin{figure}
\centering
\includegraphics[scale=0.45]{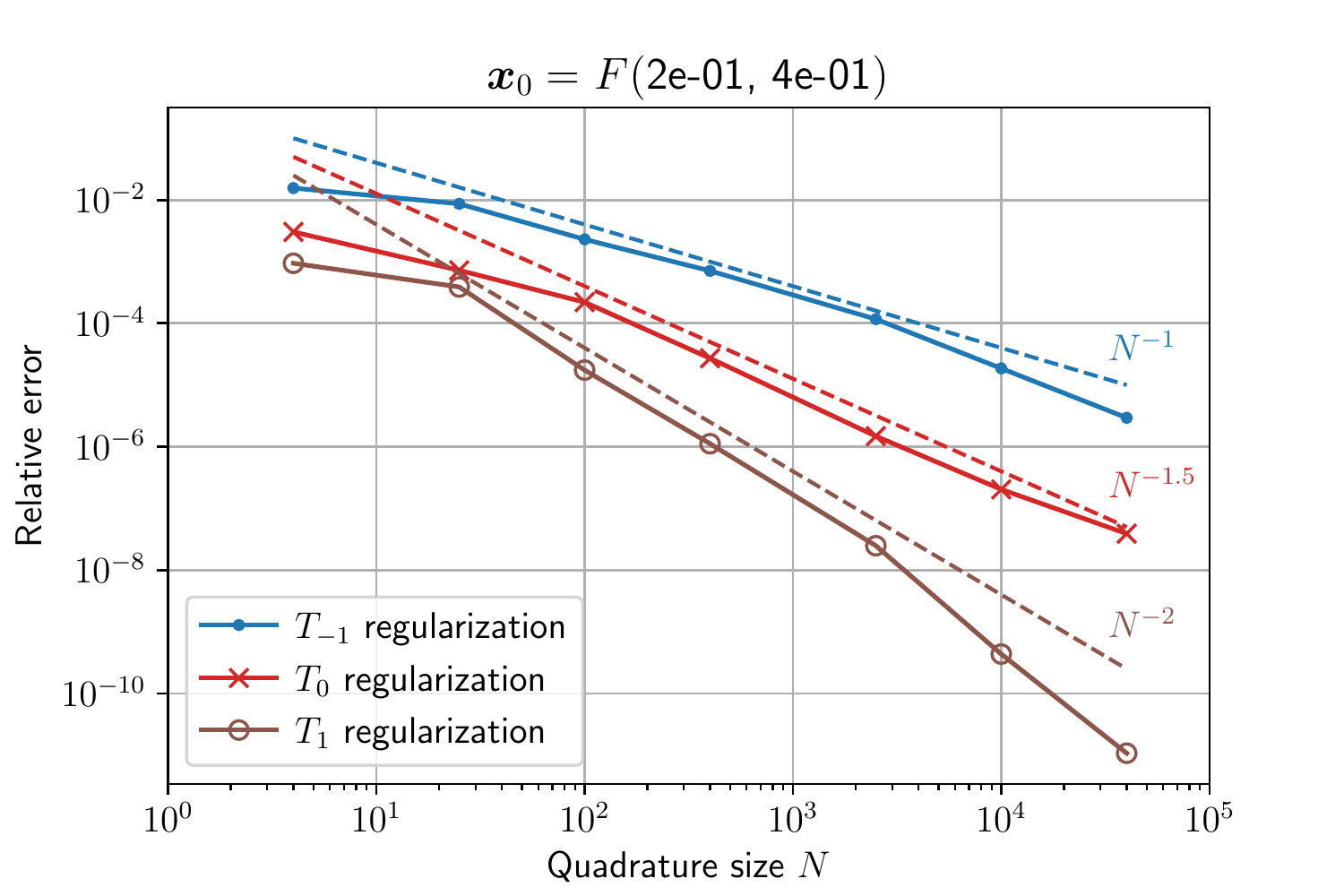}
\includegraphics[scale=0.45]{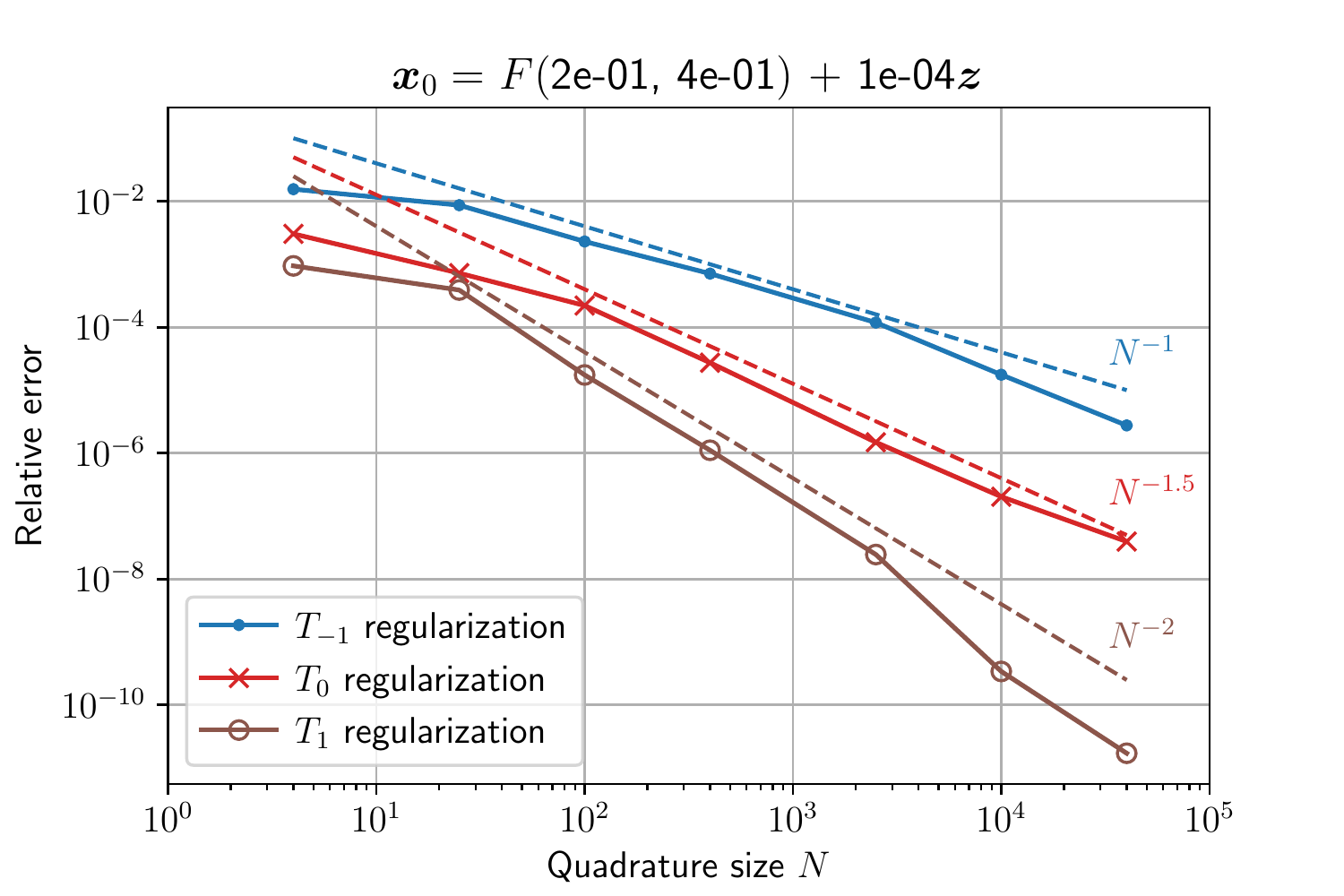}
\caption{\textit{Results for the singular (left) and near-singular (right) cases are similar. The $T_{-1}$ regularization, which is the method of \cref{sec:first-order}, converges linearly with the number of quadrature points. As we subtract more Taylor terms, the 2D integrand becomes smoother and convergence accelerates.}}
\label{fig:exp1}
\end{figure}

We now look at the case where $\bs{x}_0$ is close to an edge by taking $\bs{x}_0=F(0.5,\epsilon)$ for small values of $\epsilon$, which corresponds to points near the lower edge. In this case, it is crucial to compute the 1D integrals with the transplanted quadrature of \cref{sec:step5}. We take, again, $N=n^2$ points for the 2D integrals and $10n$ in 1D with $2\leq n\leq200$, and choose $\mu=0$ and $\nu=2\epsilon$ for the transplanted quadrature. We show the results in \cref{fig:exp2} for $\epsilon=10^{-4}$ and observe similar convergence curves.

\begin{figure}
\centering
\includegraphics[scale=0.45]{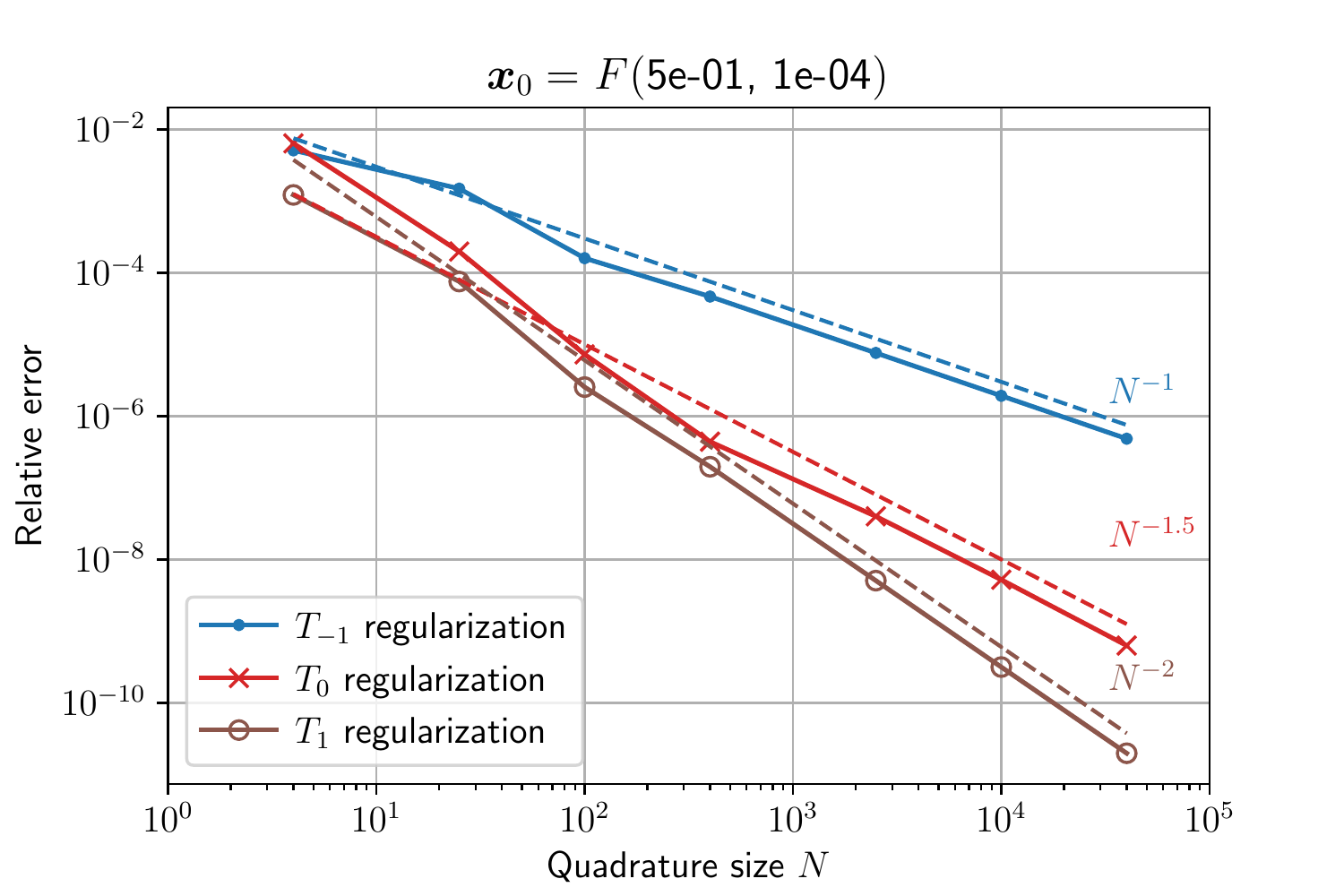}
\includegraphics[scale=0.45]{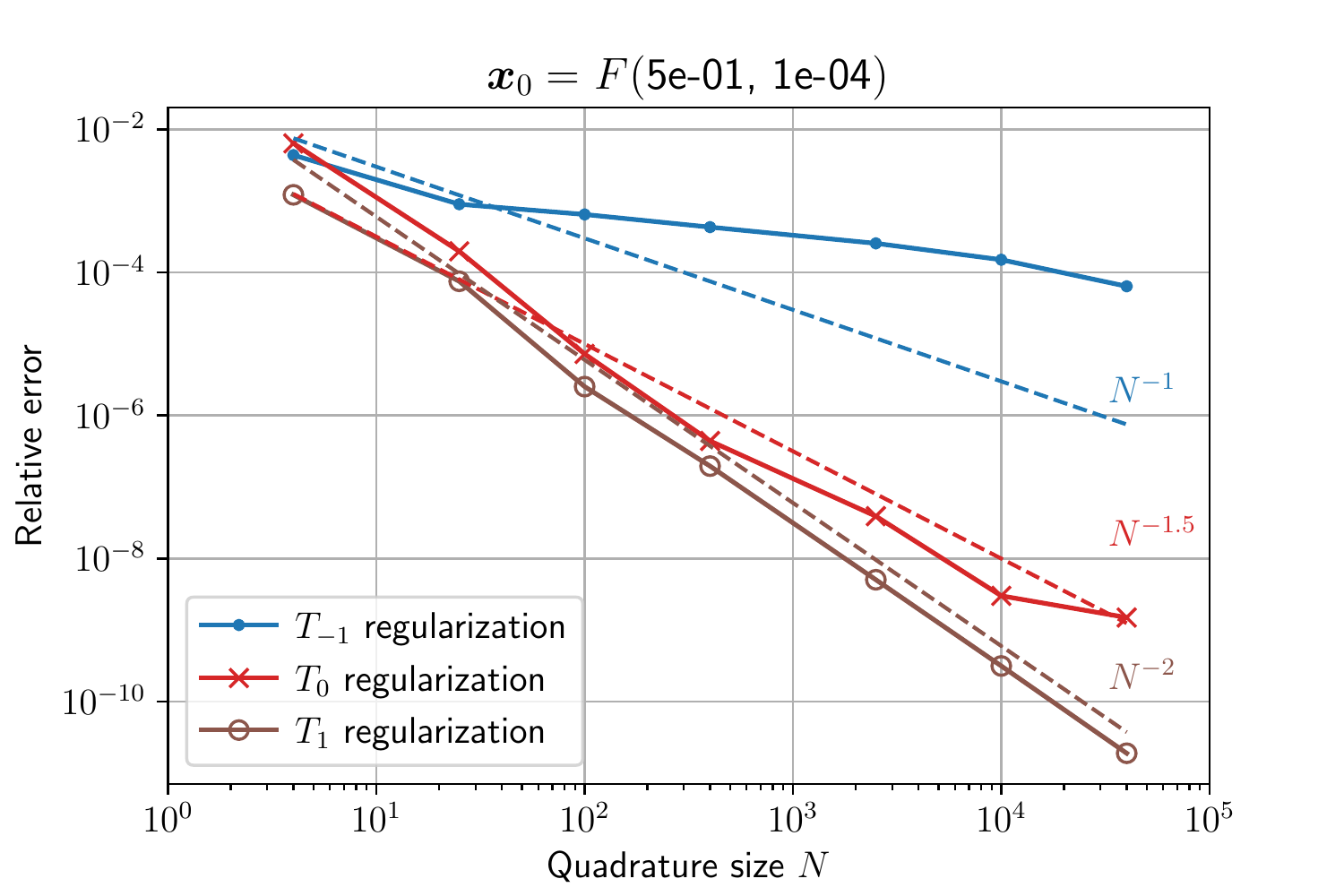}
\caption{\textit{We recover similar convergence curves when the singularity is near the lower edge (left). For this experiment, it was necessary to compute the 1D near-singular integrals with the transplanted Gauss quadrature of \cref{sec:step5}. If we employ 1D Gauss quadrature (right), the relative errors plateau at $\OO(\epsilon\log(2/\epsilon))$, $\OO(\epsilon^2\log(2/\epsilon))$ and $\OO(\epsilon^3\log(2/\epsilon))$ for $T_{-1}$, $T_0$ and $T_1$ regularization.}}
\label{fig:exp2}
\end{figure}

A code to compute these integrals in Python can be found in \cref{sec:code}. It uses $\texttt{scipy}$'s BFGS for simplicity. A code for general quadratic triangles is available on the first author's \href{https://github.com/Hadrien-Montanelli/singintpy}{GitHub} page.

\paragraph{Singular 4D integrals over two identical quadratic triangles} We explain now how we can use our method for computing integrals of the form of \cref{eq:intsing} to compute integrals over two curved triangles, which occur when solving \cref{eq:SL} with boundary elements. Suppose we are interested in computing
\begin{align}
I = \int_\mathcal{T}\int_\mathcal{T}\frac{dS(\bs{y})dS(\bs{x})}{\vert\bs{x}-\bs{y}\vert},
\label{eq:I}
\end{align}
where $\mathcal{T}$ is the triangle of \cref{fig:exp_tri} with $a=b=0.5$ and $c=1$. We map the $\bs{y}$-integral back to $\widehat{T}$,
\begin{align}
I = \int_{\mathcal{T}}\int_{\widehat{T}}\frac{\psi(\bs{\hat{y}})}{\vert\bs{x}-F(\hat{\bs{y}})\vert}dS(\hat{\bs{y}})dS(\bs{x}),
\end{align}
with $\psi(\bs{\hat{y}})=\vert J_1(\bs{\hat{y}})\times J_2(\bs{\hat{y}})\vert$. Then, we discretize it with $N$-point Gauss quadrature on triangles,
\begin{align}
I \approx I_N = \sum_{n=1}^Nw_n\psi(\hat{\bs{y}}_n)\int_{\mathcal{T}}\frac{dS(\bs{x})}{\vert\bs{x}-F(\hat{\bs{y}}_n)\vert}.
\end{align}
There remain $N$ integrals of the form of \cref{eq:intsing}, which we compute with the $N$-point quadrature methods with $T_{-1}$, $T_0$, and $T_1$ regularization that converge at the rates $\OO(N^{-1})$, $\OO(N^{-1.5})$, and $\OO(N^{-2})$. It is possible to show that the integrand of the $\bs{\hat{y}}$-integral is as regular as $\hat{y}_2\,\mrm{arcsinh}((1-\hat{y}_1)/\hat{y}_2)$, for which standard $N$-point Gauss quadrature on triangles converges at the rate $\OO(N^{-2})$. Therefore, we obtain $M$-point quadrature rules for 4D singular integrals that converge at the rates $\OO(M^{-0.5})$, $\OO(M^{-0.75})$, and $\OO(M^{-1})$. We compute the integral \cref{eq:I} with the singularity cancellation method of Sauter and Schwab \cite[Sec.~5.2.1]{sauter2011} to 8-digit accuracy and plot the errors in \cref{fig:exp3}.

\begin{figure}
\centering
\includegraphics[scale=0.5]{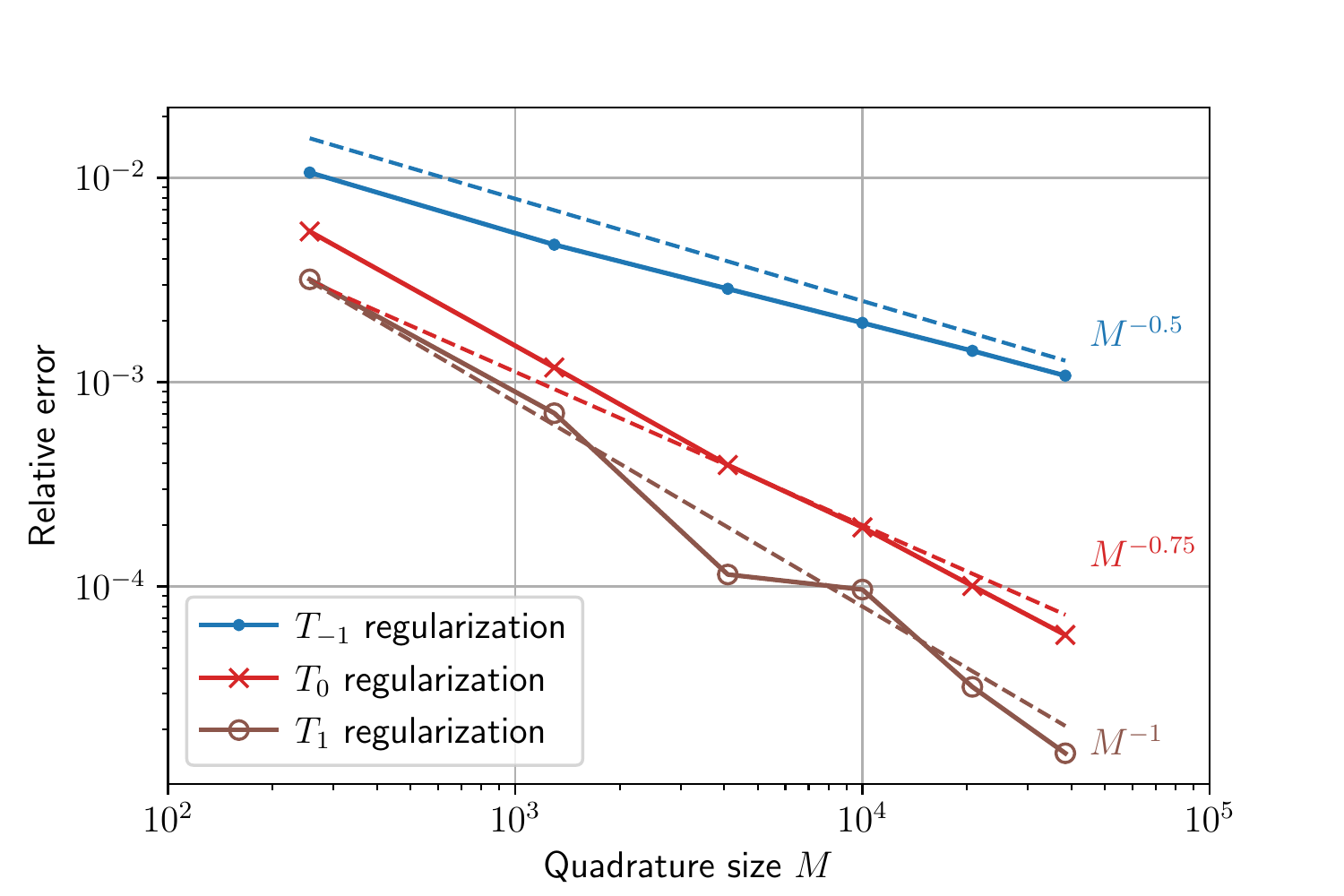}
\caption{\textit{Combining the method described in this paper with Gauss quadrature on triangles yields a new method for computing 4D singular integrals of the form of \cref{eq:I}. The method converges linearly with the total 4D number of points $M$ when using the $T_1$ regularization described in \cref{sec:high-order}.}}
\label{fig:exp3}
\end{figure}

\paragraph{Scattering by two half-spheres} Let $\Omega\subset\R^3$ be a bounded domain whose complement is connected, $\Gamma$ its boundary, and $k>0$ the wavenumber. Given an incident wave $u^i(\bs{x})$, a solution to $\Delta u + k^2u = 0$ in $\R^3$, we look for the scattered field $u^s(\bs{x})$, a solution to $\Delta u + k^2u = 0$ in $\R^3\setminus\overline{\Omega}$, satisfying the \textit{Sommerfeld radiation condition}\footnote{The Sommerfeld radiation condition, which guarantees that the scattered wave is outgoing, reads 
\begin{align}
\lim_{r\to\infty} r\left(\frac{\partial u}{\partial r} - iku\right)=0, \quad \text{uniformly with respect to}\;\,r=\vert\bs{x}\vert.
\end{align}} and such that $u^i(\bs{x}) + u^s(\bs{x}) = 0$ on $\Gamma=\partial\Omega$. As said in the introduction, assuming that $k^2$ is not an eigenvalue of $-\Delta$ in $\Omega$, this leads to \cite[Thm.~3.28]{colton1983},
\begin{align}
\frac{1}{4\pi}\int_{\Gamma}\frac{e^{ik\vert\bs{x}-\bs{y}\vert}}{\vert\bs{x}-\bs{y}\vert}\varphi^s(\bs{y})d\Gamma(\bs{y}) = -u^i(\bs{x}), \quad \bs{x}\in\Gamma.
\label{eq:SL2}
\end{align}
Once the equation \cref{eq:SL2} is solved for $\varphi^s$, the scattered field $u^s$ is given by 
\begin{align}
u^s(\bs{x}) = \frac{1}{4\pi}\int_{\Gamma}\frac{e^{ik\vert\bs{x}-\bs{y}\vert}}{\vert\bs{x}-\bs{y}\vert}\varphi^s(\bs{y})d\Gamma(\bs{y}), \quad \bs{x}\in\R^3\setminus\Omega.
\end{align}

Of particular interest is the \textit{far-field pattern} $u_\infty$ defined on the unit sphere $\Stwo$ via \cite[Thm.~2.6]{colton2013}
\begin{align}
u^s(\bs{x}) = \frac{e^{ikr}}{r}u_\infty\left(\frac{\bs{x}}{r}\right) + \OO\left(\frac{1}{r^2}\right), \quad r=\vert\bs{x}\vert\to\infty,
\label{eq:u_s}
\end{align}
with integral representation \cite[Thm.~3.14]{colton2013}
\begin{align}
u_\infty(\bar{\bs{x}}) = \frac{1}{4\pi}\int_{\Gamma}e^{-ik\bar{\bs{x}}\cdot\bs{y}}\varphi^s(\bs{y})d\Gamma(\bs{y}), \quad \bar{\bs{x}}\in\Stwo.
\label{eq:u_inf}
\end{align}

We discretize \cref{eq:SL2} using a boundary element method with quadratic basis functions ($p=2$) and quadratic triangles ($q=2$). This yields the computation of integrals of the form \cite[Chap.~5]{sauter2011}
\begin{align}
I = \frac{1}{4\pi}\int_{\mathcal{T}_\ell}\int_{\mathcal{T}_{\ell'}}\frac{e^{ik\vert\bs{x}-\bs{y}\vert}}{\vert\bs{x}-\bs{y}\vert}\varphi_{j'}(F_{\ell'}^{-1}(\bs{y}))dS(\bs{y})\varphi_j(F_\ell^{-1}(\bs{x}))dS(\bs{x}),
\label{eq:I_2}
\end{align}
where the $\varphi_j$'s are the basis functions defined in \cref{eq:basis_func}, and $\mathcal{T}_\ell$ and $\mathcal{T}_{\ell'}$ are two quadratic triangles. (The integral $I$ in \cref{eq:I_2} is singular/near-singular when $\mathcal{T}_\ell$ and $\mathcal{T}_{\ell'}$ are identical/close---when the triangles are far apart, 4D Gauss quadrature may be used with exponential convergence.) 

As in the previous numerical experiment, we first map the inner integral back to $\widehat{T}$ and then discretize it with $N$-point Gauss quadrature on triangles,
\begin{align}
I \approx I_N = \frac{1}{4\pi}\sum_{n=1}^Nw_n\psi_{j'}(\hat{\bs{y}}_n)\int_{\mathcal{T}_\ell}\frac{e^{ik\vert\bs{x}-F_{\ell'}(\hat{\bs{y}}_n)\vert}}{\vert\bs{x}-F_{\ell'}(\hat{\bs{y}}_n)\vert}\varphi_j(F_\ell^{-1}(\bs{x}))dS(\bs{x}),
\end{align}
where $\psi_{j'}$ includes the Jacobian. We are left with the computations of $N$ integrals of the form
\begin{align}
\int_{\mathcal{T}_\ell}\frac{e^{ik\vert\bs{x}-\bs{x}_n\vert}-1}{\vert\bs{x}-\bs{x}_n\vert}\varphi_j(F_\ell^{-1}(\bs{x}))dS(\bs{x}) + \int_{\mathcal{T}_\ell}\frac{\varphi_j(F_\ell^{-1}(\bs{x}))}{\vert\bs{x}-\bs{x}_n\vert}dS(\bs{x}).
\end{align}
The first integral, whose integrand has bounded first derivatives, is also discretized with $N$-point Gauss quadrature (convergence rate $\OO(N^{-1.5})$), while the second one is discretized with the method described in this paper with $T_{-1}$ regularization (convergence rate $\OO(N^{-1})$); we take $N=16^2$ points.

We take $k\in\{2\pi,4\pi,8\pi,16\pi\}$ and consider the scattering of a plane wave $u^i(r,\theta)=e^{ikr\cos\theta}$ by two half-spheres of radius $1$ centered at $(0,0,\pm\delta)$; see \cref{fig:exp4a}. We first compute the solution for~$\delta=0$, which corresponds to the scattering of the incident wave by the unit sphere---we solve \cref{eq:SL2} and evaluate \cref{eq:u_inf} for an increasing number of quadratic elements. We plot the relative $\infty$-norm error in the far-field pattern in \cref{fig:exp4b} (left); the exact far-field pattern is \cite[Eqn.~3.32]{colton2013}
\begin{align}
u_\infty(\theta) = \frac{i}{k}\sum_{n=0}^\infty(2n+1)\frac{j_n(k)}{h_n^{(1)}(k)}P_n(\cos\theta), \quad \theta\in[0,2\pi], 
\label{eq:u_inf_ex}
\end{align}
with Legendre polynomials $P_n$, and spherical Bessel and Hankel functions $j_n$ and $h_n^{(1)}$. We observe quartic \textit{superconvergence} as the mesh size $h\to0$; cubic convergence was expected.\footnote{For boundary elements of degree $(p,q)$ with a mesh of size $h$, the error in the numerical far-field is bounded by
\begin{align}
\vert u_\infty(\theta) - u_{\infty,h}(\theta)\vert \leq c\left\{h^{2(p+1)+1}\Vert\varphi^s\Vert_{H^{p+1}(\Gamma)} + h^{q+1}\Vert\varphi^s\Vert_{L^2(\Gamma)}\right\}.
\end{align}
Results of this form go back to \cite{nedelec1976}; see also \cite[Chap.~8]{sauter2011}. For the sphere, $\OO(h^{q+1})$ seems to improve to $\OO(h^{2q})$.}

\begin{figure}
\begin{tabular}{cc}
\includegraphics[scale=0.225]{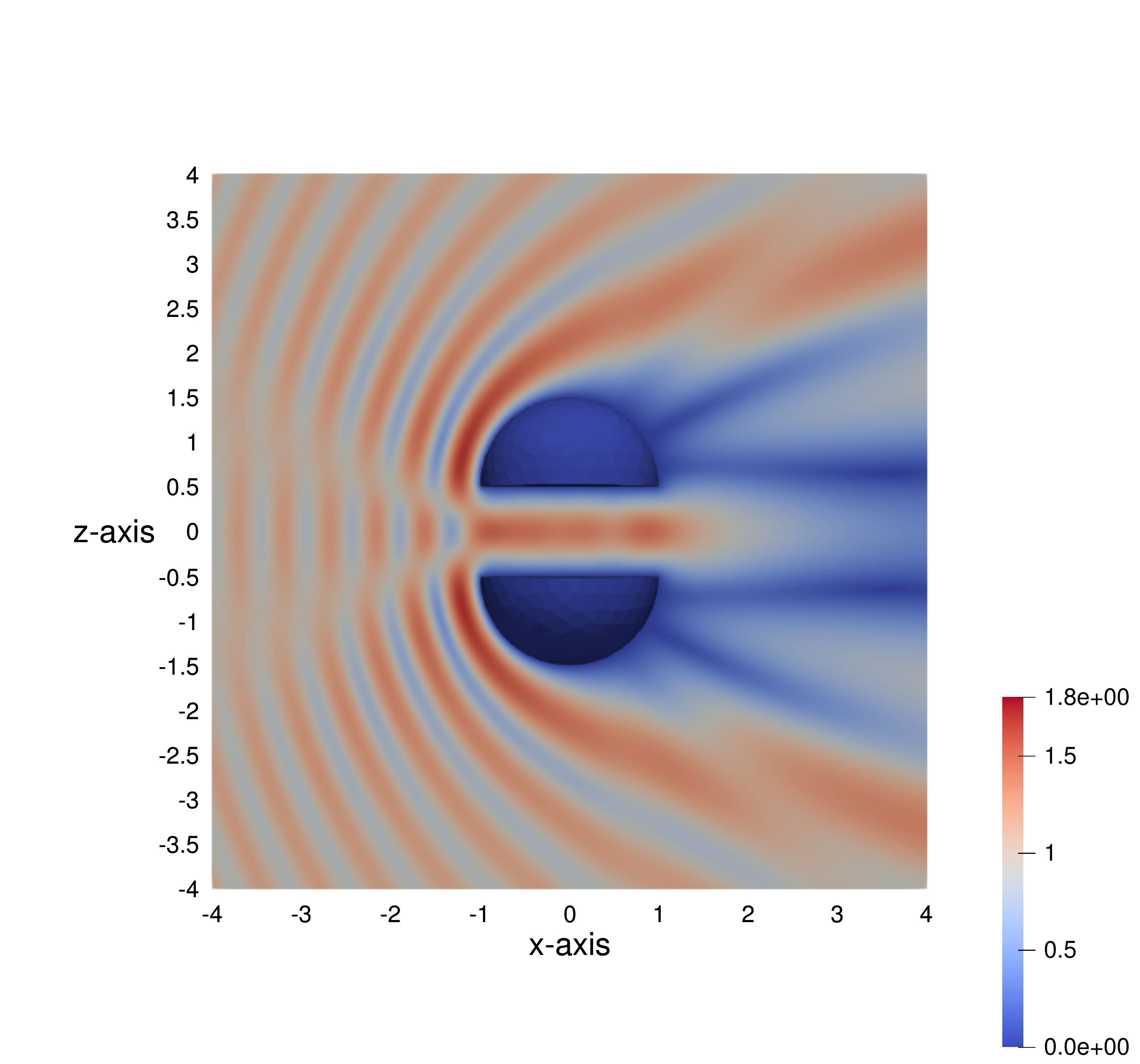} &
\includegraphics[scale=0.225]{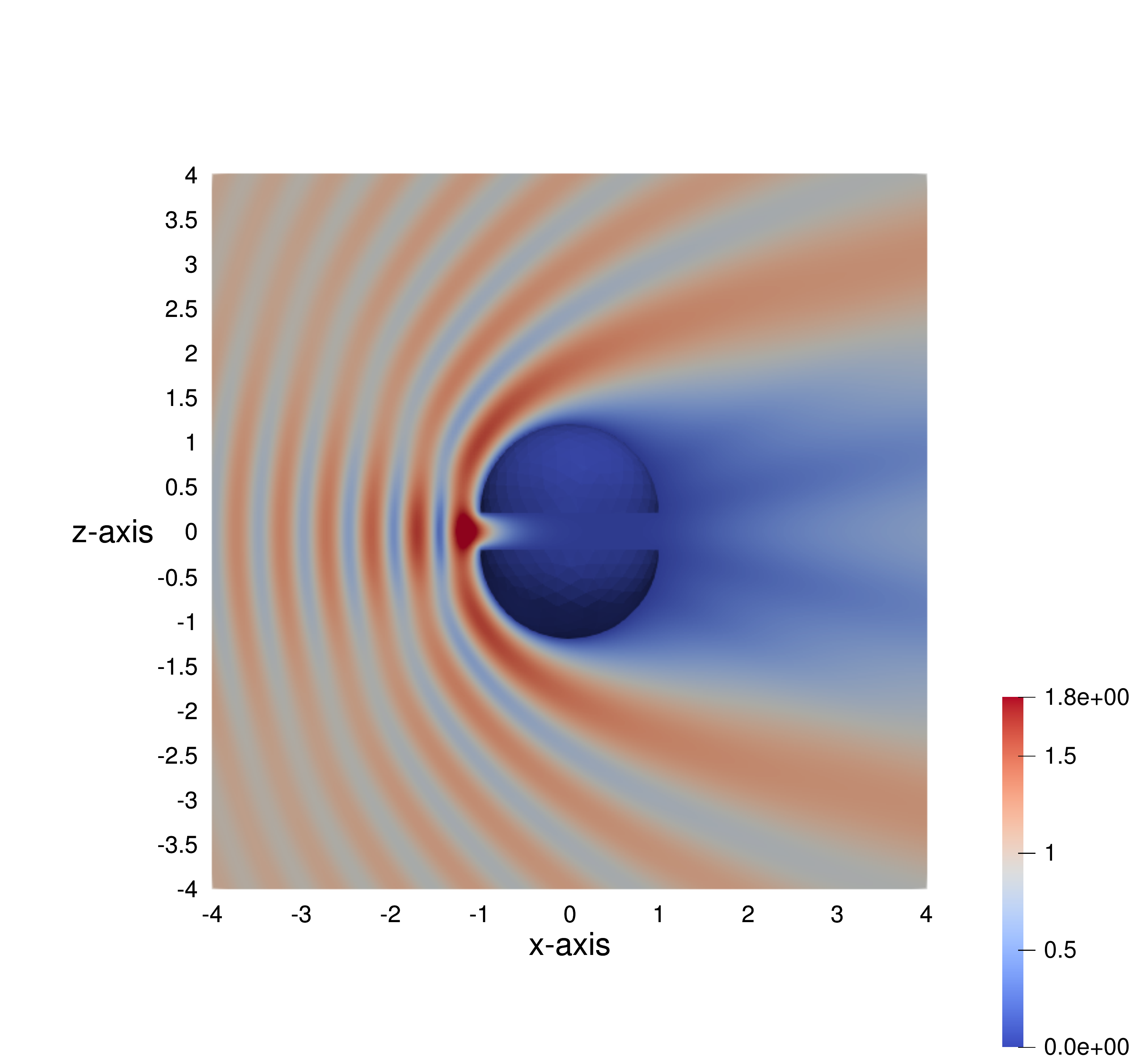} \\
\includegraphics[scale=0.225]{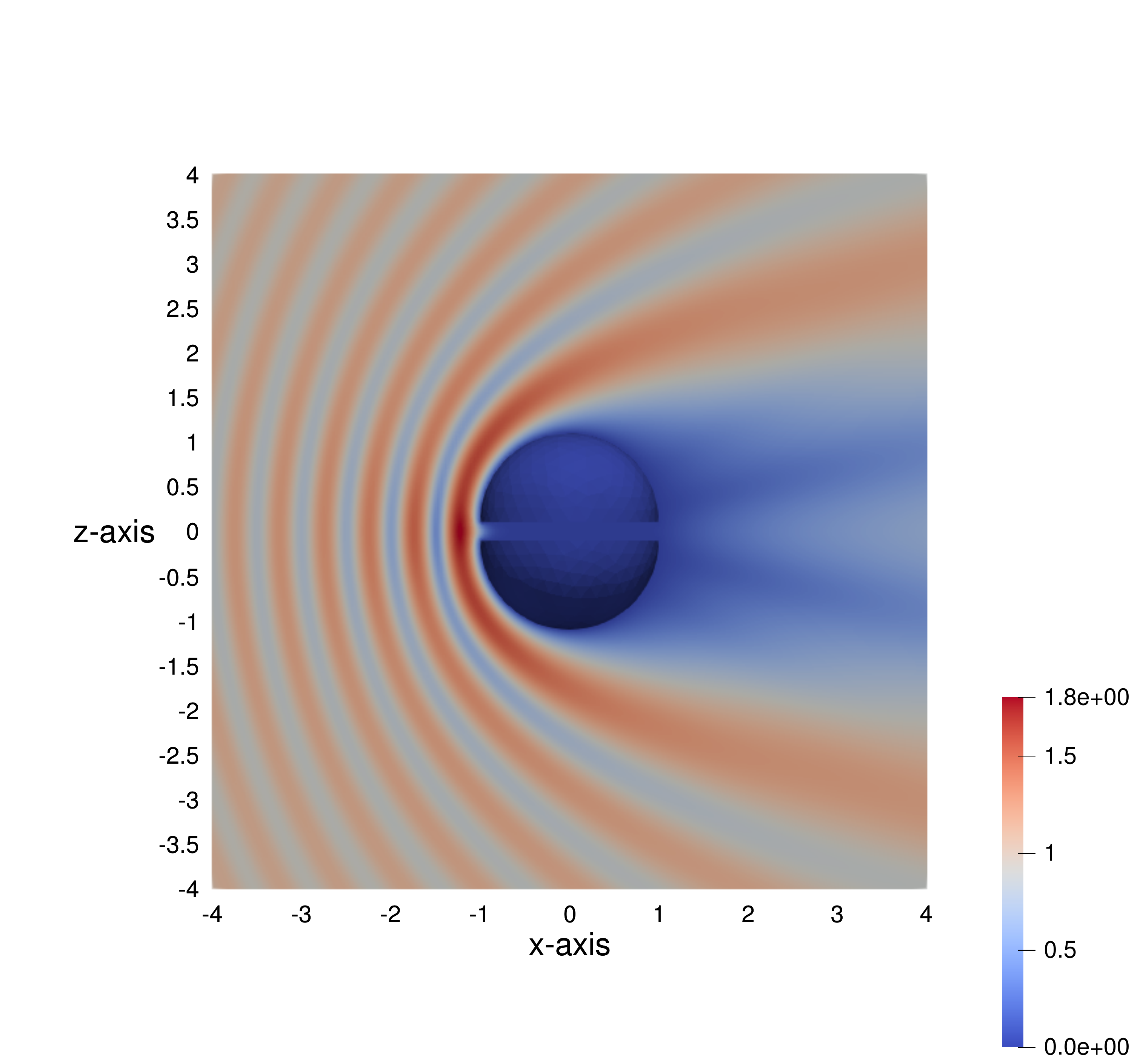} &
\includegraphics[scale=0.225]{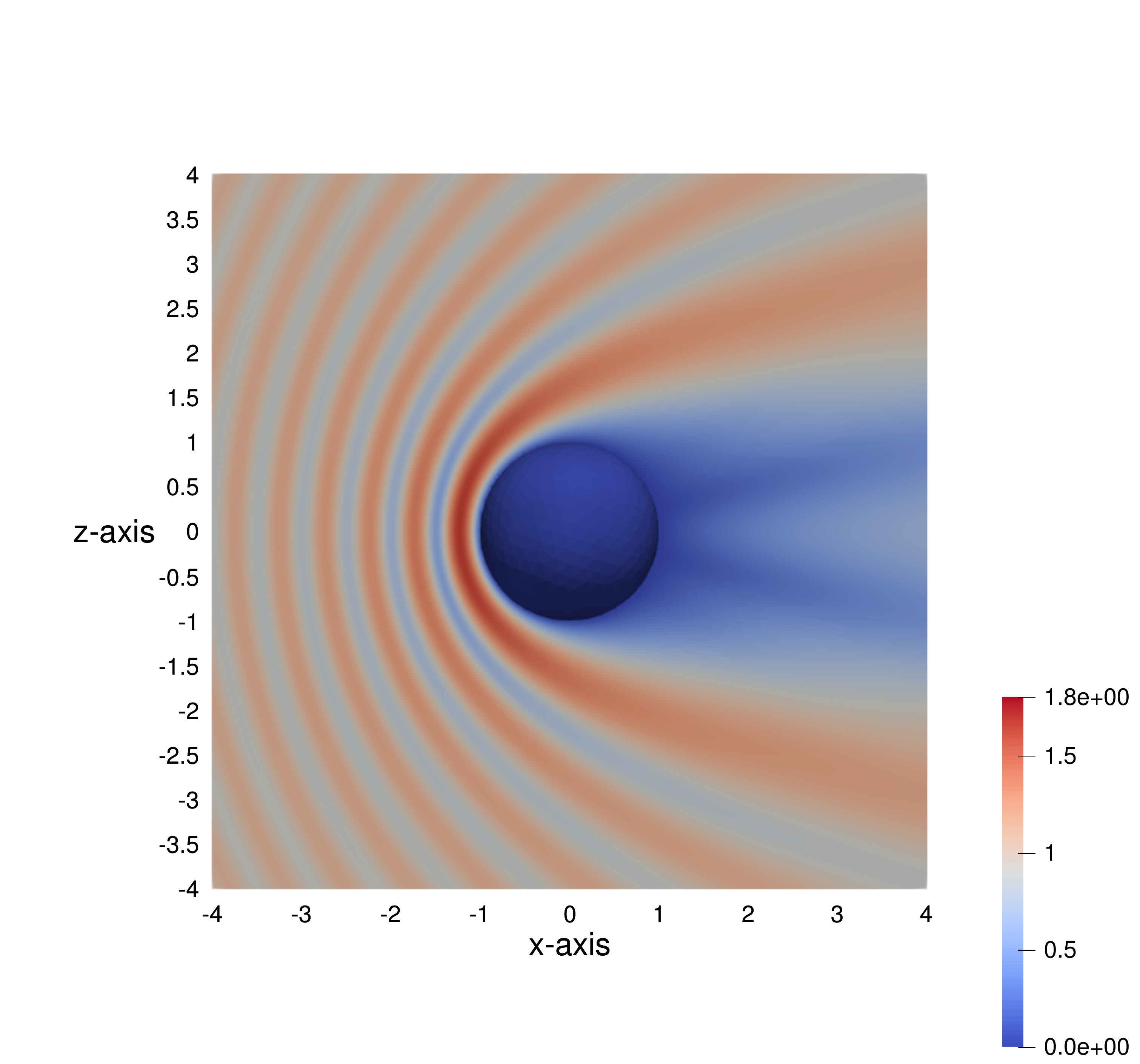} 
\end{tabular}
\caption{\textit{The incident plane wave $u^i(r,\theta)=e^{ikr\cos\theta}$ travels along the $x$-axis towards $x>0$ and is scattered by two half-spheres of radius $1$ centered at $(0,0,\pm\delta)$. We plot the amplitude of $u^i+u^s$ for $k=2\pi$ and $\delta=0.5$ (top left), $\delta=0.2$ (top right), $\delta=0.1$ (bottom left), and $\delta=0$ (bottom right). For $\delta=0.5$, the distance between the two half-spheres, $2\delta=1$, equals the wavelength $\lambda=2\pi/k=1$; the incident wave ``sees the gap'' and is able to go through it. This not the case when $\delta<\lambda$.}}
\label{fig:exp4a}
\end{figure}

\begin{figure}
\centering
\includegraphics[scale=0.45]{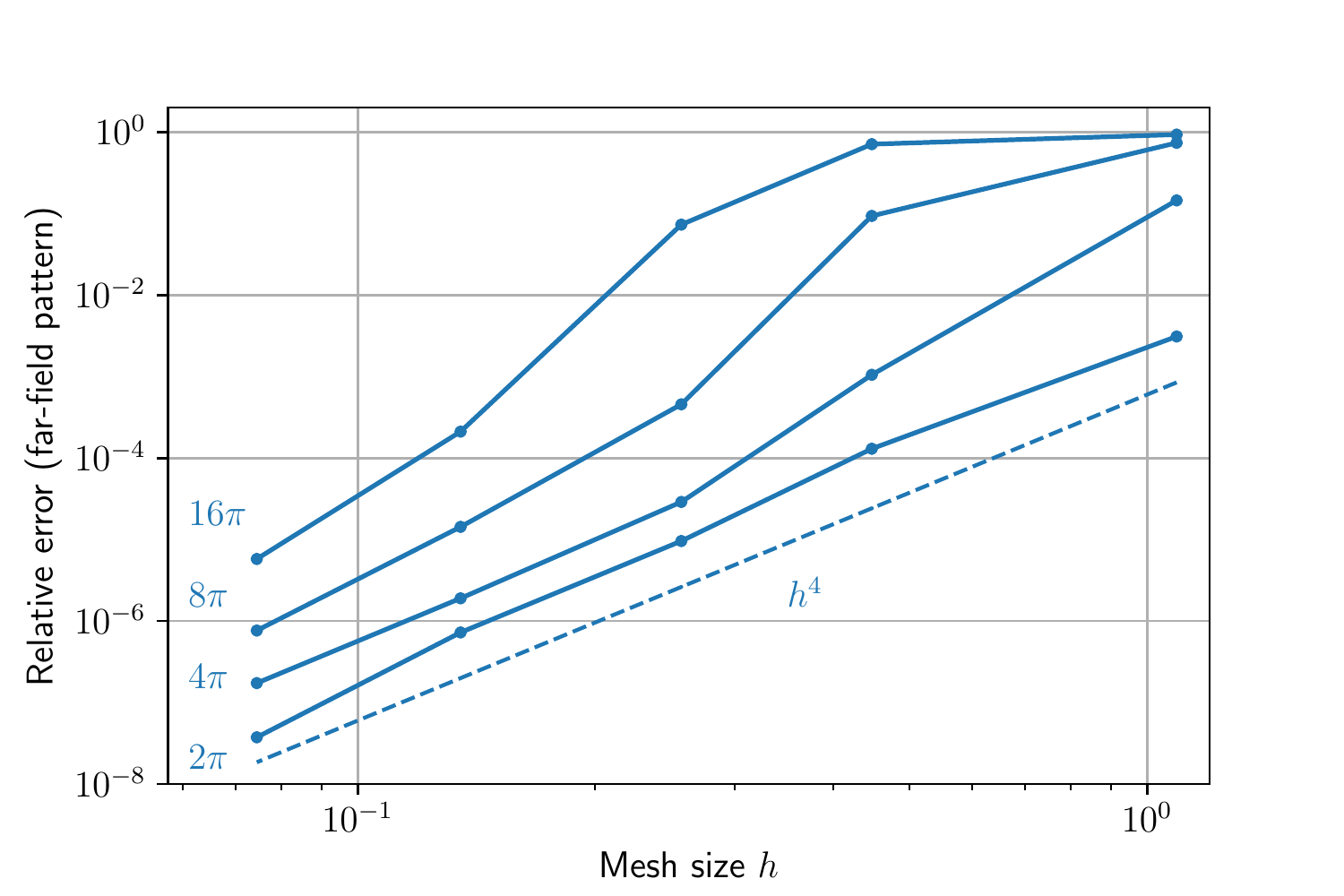}
\includegraphics[scale=0.45]{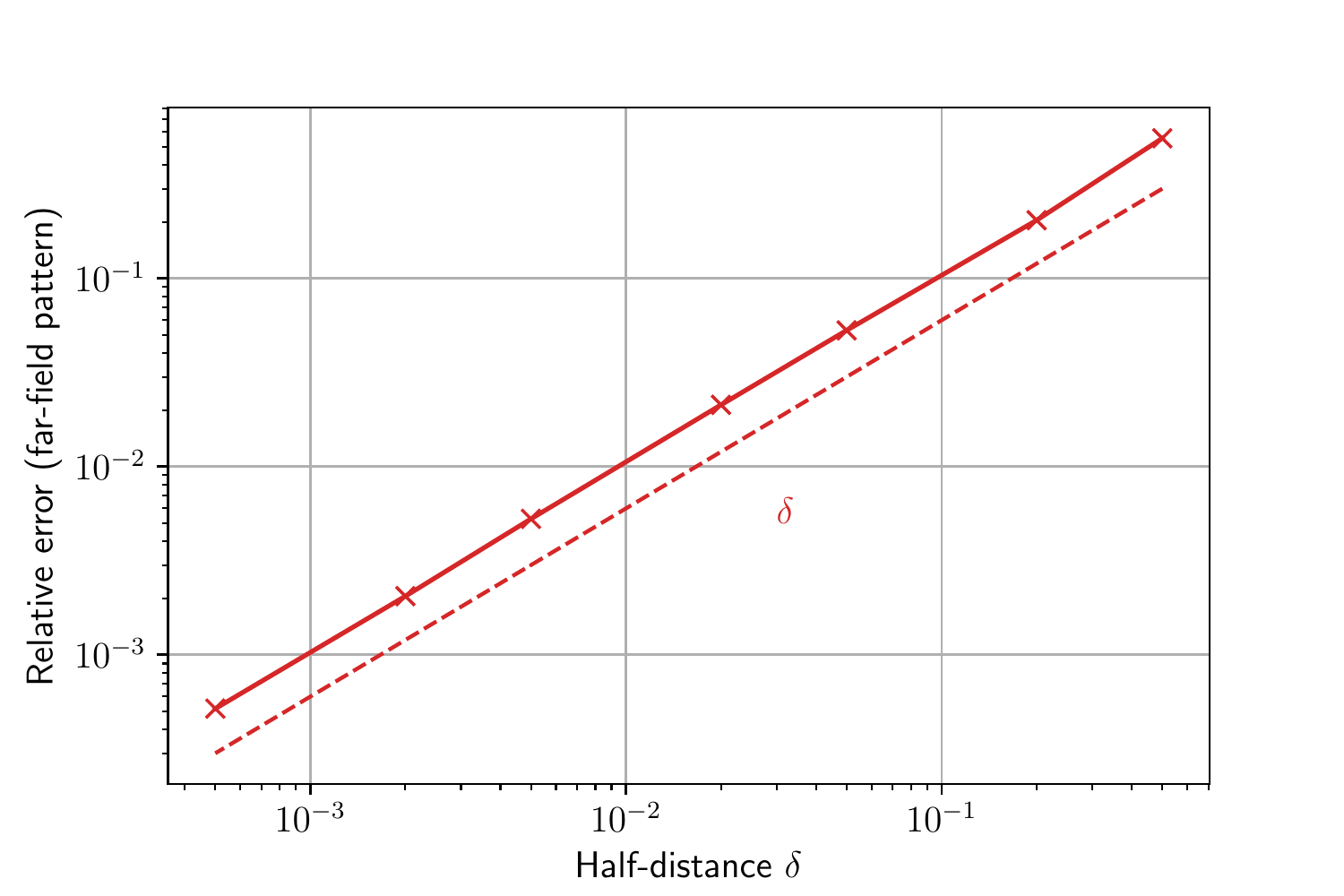}
\caption{\textit{For $\delta=0$, the incident wave is scattered by the unit sphere; the numerical far-field converges quartically to the exact far-field \cref{eq:u_inf_ex} as the mesh size $h\to0$ (left). For $\delta>0$, it is scattered by two half-spheres centered at $(0,0,\pm\delta)$ and separated by a distance $2\delta$; the numerical far-field for $k=2\pi$ is computed for each $\delta$ with $h\approx10^{-1}$. It converges linearly to \cref{eq:u_inf_ex} as $\delta\to0$ (right).}}
\label{fig:exp4b}
\end{figure}

We now compute the solution for $k=2\pi$ and small values of $\delta>0$---for each of these values, we solve \cref{eq:SL2} and evaluate \cref{eq:u_inf} for a mesh size $h\approx10^{-1}$, which gives about five digits of~accuracy. We observe in \cref{fig:exp4b} (right) linear convergence of these far-field patterns to \cref{eq:u_inf_ex} as $\delta\to0$, in agreement with \cite{delourme2013, delourme2012}. This experiment is particularly challenging when $\delta\ll h$. The method of Sauter and Schwab \cite[Sec.~5]{sauter2011}, designed for elements that touch each other, cannot be applied to elements facing each other across the gap---these would simply be computed with Gauss quadrature, yielding very inaccurate results. We obtained accurate results for values of $\delta$ as small as~$5\times10^{-4}$.

We wrap up this section with a few words about implementation. We added our novel method for computing singular and near-singular integrals---as well as new features to handle high-order boundary elements and basis functions---to the C++ \href{http://leprojetcastor.gitlab.labos.polytechnique.fr/castor/}{\texttt{castor}} library of \'{E}cole Polytechnique, whose lead developer is the second author. The \texttt{castor} library provides tools to create and manipulate matrices \textit{\`a la} MATLAB, and uses an optimized BLAS library for fast linear algebra computations. The finest meshes ($h\approx7.45\times10^{-2}$) yield dense matrices of size $\approx32,000\times32,000$---we employed hierarchical matrices for compression, and, to solve the resulting linear systems, GMRES \cite{saad1986} preconditioned with a hierarchical $LU$ factorization at a lower precision. Finally, we utilized Gmsh~\cite{geuzaine2009} to generate quadratic triangular elements. The computations were carried out on an Intel Xeon Gold 6154 processor (3.00 GHz, 36 cores) with 512 GB of RAM.
 
\section{Discussion}

We presented in this work algorithms for computing singular and near-singular integrals on curved triangular elements of the form of \cref{eq:intsing}. These are particularly relevant to the single-layer potential formulation of the 3D Helmholtz Dirichlet problem in the presence of close obstacles. These are also useful for evaluating the boundary element solution close to the surface over which the single-layer potential is defined. Our methodology is based on singularity projection, singularity subtraction, the continuation approach, and transplanted Gauss quadrature.

We provided several numerical examples for quadratic basis functions and triangles in \cref{sec:numerics} but our method works for functions and triangles of any degree $p\geq0$ and $q\geq1$. Moreover, the extension to quadrilateral elements is straightforward. Finally, we focused on functions $F$ and $\varphi$ that arise with boundary elements, but our techniques apply to any smooth functions $F$ and $\varphi$. 

There are many ways in which this work could be profitably continued. For instance, one could consider singularity subtraction with Taylor-like asymptotic expansions of order $\ell\geq4$, which would lead to $N$-point quadrature rules that converge at the rates $\OO(N^{-(\ell+1)/2})$. One could also extend our procedure to strongly singular kernels, which will be the subject of a forthcoming publication. This would allow us to solve the Dirichlet problem with the \textit{double-layer potential} via \cite[Thm.~3.15]{colton1983}
\begin{align}
\int_{\Gamma}\frac{\partial G(\bs{x},\bs{y})}{\partial n(\bs{y})}\varphi(\bs{y})d\Gamma(\bs{y}) + \frac{\varphi(\bs{x})}{2}= u_D(\bs{x}), \quad \bs{x}\in\Gamma.
\end{align}

\appendix

\section{Taylor coefficients of $R^2$}\label{sec:taylor_R2} 
The Taylor expansion of $R^2=\vert F(\bs{\hat{x}}) - \bs{x}_0\vert^2$ is
\begin{align}
\vert F(\bs{\hat{x}}) - \bs{x}_0\vert^2 & = \underbrace{\vert J_0\delta\bs{\hat{x}}\vert^2 + h^2}_{\OO(\delta\hat{x}^2)} + \underbrace{h\sum_{i=1}^3a_i\delta\hat{x}_1^{3-i}\delta\hat{x}_2^{i-1} + \sum_{i=1}^4c_i\delta \hat{x}_1^{4-i}\delta\hat{x}_2^{i-1}}_{\OO(\delta\hat{x}^3)} \nonumber \\
& + \underbrace{h\sum_{i=1}^4b_i\delta\hat{x}_1^{4-i}\delta\hat{x}_2^{i-1} + \sum_{i=1}^5d_i\delta\hat{x}_1^{5-i}\delta\hat{x}_2^{i-1}}_{\OO(\delta\hat{x}^4)} + \OO(\delta\hat{x}^5).
\end{align}
The coefficients are:
\begin{align}
& a_1 = \bs{e}_h\cdot F_{\hat{x}_1\hat{x}_1}, \nonumber\eqvspp
& a_2 = 2\bs{e}_h\cdot F_{\hat{x}_1\hat{x}_2}, \nonumber\eqvspp
& a_3 = \bs{e}_h\cdot F_{\hat{x}_2\hat{x}_2}, \nonumber\eqvsp
& b_1 = \frac{1}{3}\bs{e}_h\cdot F_{\hat{x}_1\hat{x}_1\hat{x}_1}, \nonumber\eqvsp
& b_2 = \bs{e}_h\cdot F_{\hat{x}_1\hat{x}_1\hat{x}_2}, \nonumber\eqvspp
& b_3 = \bs{e}_h\cdot F_{\hat{x}_1\hat{x}_2\hat{x}_2}, \nonumber\eqvsp
& b_4 = \frac{1}{3}\bs{e}_h\cdot F_{\hat{x}_2\hat{x}_2\hat{x}_2}, \nonumber\eqvsp
& c_1 = F_{\hat{x}_1}\cdot F_{\hat{x}_1\hat{x}_1}, \eqvspp
& c_2 = 2F_{\hat{x}_1}\cdot F_{\hat{x}_1\hat{x}_2} + F_{\hat{x}_2}\cdot F_{\hat{x}_1\hat{x}_1}, \nonumber\eqvspp
& c_3 = 2F_{\hat{x}_2}\cdot F_{\hat{x}_1\hat{x}_2} + F_{\hat{x}_1}\cdot F_{\hat{x}_2\hat{x}_2}, \nonumber\eqvspp
& c_4 = F_{\hat{x}_2}\cdot F_{\hat{x}_2\hat{x}_2}, \nonumber\eqvsp
& d_1 = \frac{1}{3}F_{\hat{x}_1}\cdot F_{\hat{x}_1\hat{x}_1\hat{x}_1} + \frac{1}{4}F_{\hat{x}_1\hat{x}_1}\cdot F_{\hat{x}_1\hat{x}_1}, \nonumber\eqvsp
& d_2 = \frac{1}{3}F_{\hat{x}_2}\cdot F_{\hat{x}_1\hat{x}_1\hat{x}_1} + F_{\hat{x}_1}\cdot F_{\hat{x}_1\hat{x}_1\hat{x}_2} + F_{\hat{x}_1\hat{x}_1}\cdot F_{\hat{x}_1\hat{x}_2}, \nonumber\eqvsp
& d_3 = \frac{1}{2}F_{\hat{x}_1\hat{x}_1}\cdot F_{\hat{x}_2\hat{x}_2} + F_{\hat{x}_1}\cdot F_{\hat{x}_1\hat{x}_2\hat{x}_2} + F_{\hat{x}_2}\cdot F_{\hat{x}_1\hat{x}_1\hat{x}_2} 
+ F_{\hat{x}_1\hat{x}_2}\cdot F_{\hat{x}_1\hat{x}_2}, \nonumber\eqvsp
& d_4 = \frac{1}{3}F_{\hat{x}_1}\cdot F_{\hat{x}_2\hat{x}_2\hat{x}_2} + F_{\hat{x}_2}\cdot F_{\hat{x}_1\hat{x}_2\hat{x}_2} + F_{\hat{x}_2\hat{x}_2}\cdot F_{\hat{x}_1\hat{x}_2}, \nonumber\eqvsp
& d_5 = \frac{1}{3}F_{\hat{x}_2}\cdot F_{\hat{x}_2\hat{x}_2\hat{x}_2} + \frac{1}{4}F_{\hat{x}_2\hat{x}_2}\cdot F_{\hat{x}_2\hat{x}_2}, \nonumber
\end{align}
where $\bs{e}_h$ is the unit vector such that $F(\bs{\hat{x}}_0) - \bs{x}_0 = h\bs{e}_h$. Note that the $c_i$'s and the $d_i$'s are those derived in \cite{aliabadi1985}---our contributions are the $a_i$'s and the $b_i$'s. Partial derivatives are evaluated at $\bs{\hat{x}}_0$.

\newpage 
\section{Taylor coefficients of $\psi R^{-1}$} \label{sec:taylor_psiRn1}
The expansion of $\psi R^{-1}=v\vert F(\bs{\hat{x}}) - \bs{x}_0\vert^{-1}$ is
\begin{align}
\psi(\bs{\hat{x}})\vert F(\bs{\hat{x}}) - \bs{x}_0\vert^{-1} = T_{-1}(\bs{\hat{x}};h) + T_0(\bs{\hat{x}};h) + T_1(\bs{\hat{x}};h) + \OO(\delta\hat{x}^2).
\end{align}
The coefficients are (the $g_i$'s and the $h_i$'s are from \cite{aliabadi1985}, the $e_i$'s and $f_i$'s appear to be new):
\begin{align}
& e_1 = -\frac{1}{2}\left[a_1\psi_{\hat{x}_1} + b_1\psi_0\right], \nonumber\\
& e_2 = -\frac{1}{2}\left[a_2\psi_{\hat{x}_1} + a_1\psi_{\hat{x}_2} +  b_2\psi_0\right], \nonumber\\
& e_3 = -\frac{1}{2}\left[a_3\psi_{\hat{x}_1} + a_2\psi_{\hat{x}_2} +  b_3\psi_0\right], \nonumber\\
& e_4 = -\frac{1}{2}\left[a_3\psi_{\hat{x}_2} + b_4\psi_0\right], \nonumber\\
& f_1 = \frac{3}{8}\psi_0\left[a_1^2\right], \nonumber\\
& f_2 = \frac{3}{8}\psi_0\left[2a_1a_2\right], \nonumber\\
& f_3 = \frac{3}{8}\psi_0\left[a_2^2 + 2a_1a_3\right], \nonumber\\
& f_4 = \frac{3}{8}\psi_0\left[2a_2a_3\right], \nonumber\\
& f_5 = \frac{3}{8}\psi_0\left[a_3^2\right], \nonumber\\
& g_1 = -\frac{1}{2}\left[c_1\psi_{\hat{x}_1} + d_1\psi_0\right], \nonumber\\
& g_2 = -\frac{1}{2}\left[c_2\psi_{\hat{x}_1} + c_1\psi_{\hat{x}_2} +  d_2\psi_0\right], \\
& g_3 = -\frac{1}{2}\left[c_3\psi_{\hat{x}_1} + c_2\psi_{\hat{x}_2} +  d_3\psi_0\right], \nonumber\\
& g_4 = -\frac{1}{2}\left[c_4\psi_{\hat{x}_1} + c_3\psi_{\hat{x}_2} +  d_4\psi_0\right], \nonumber\\
& g_5 = -\frac{1}{2}\left[c_4\psi_{\hat{x}_2} + d_5\psi_0\right], \nonumber\\
& h_1 = \frac{3}{8}\psi_0\left[c_1^2\right], \nonumber\\
& h_2 = \frac{3}{8}\psi_0\left[2c_1c_2\right], \nonumber\\
& h_3 = \frac{3}{8}\psi_0\left[c_2^2 + 2c_1c_3\right], \nonumber\\
& h_4 = \frac{3}{8}\psi_0\left[2c_1c_4 + 2c_2c_3\right], \nonumber\\
& h_5 = \frac{3}{8}\psi_0\left[c_3^2 + 2c_2c_4\right], \nonumber\\
& h_6 = \frac{3}{8}\psi_0\left[2c_3c_4\right], \nonumber\\
& h_7 = \frac{3}{8}\psi_0\left[c_4^2\right]. \nonumber
\end{align}

\section{Computation of $I_0$ and $I_1$}\label{sec:I0_I1}
Using the continuation approach, we obtain
\begin{align}
I_0(h) & = \sum_{j=1}^3\hat{s}_j\int_{\partial\widehat{T}_j-\bs{\hat{x}}_0}\psi'_0(\bs{\hat{x}})\frac{\dsp\vert J_0\bs{\hat{x}}\vert\sqrt{\vert J(\bs{\hat{x}}_0)\bs{\hat{x}}\vert^2 + h^2}-h^2\mrm{asinh}\left(\frac{\vert J_0\bs{\hat{x}}\vert}{h}\right)}{2\vert J_0\bs{\hat{x}}\vert^3}ds(\bs{\hat{x}}) \nonumber \eqvspp
& - \frac{h\psi_0}{2}\sum_{i=1}^3a_i\sum_{j=1}^3\hat{s}_j\int_{\partial\widehat{T}_j-\bs{\hat{x}}_0}\hat{x}_1^{3-i}\hat{x}_2^{i-1}\frac{\vert J_0\bs{\hat{x}}\vert^2+2h\left(h- \sqrt{\vert J_0\bs{\hat{x}}\vert^2 + h^2}\right)}{\vert J_0\bs{\hat{x}}\vert^4\sqrt{\vert J_0\bs{\hat{x}}\vert^2 + h^2}}ds(\bs{\hat{x}}) \label{eq:I0_vtx} \eqvspp
& - \frac{\psi_0}{2}\sum_{i=1}^4c_i\sum_{j=1}^3\hat{s}_j\int_{\partial\widehat{T}_j-\bs{\hat{x}}_0}\hat{x}_1^{4-i}\hat{x}_2^{i-1}\frac{\dsp\frac{3h^2\vert J_0\bs{\hat{x}}\vert + \vert J_0\bs{\hat{x}}\vert^3}{\sqrt{\vert J_0\bs{\hat{x}}\vert^2 + h^2}}-3h^2\mrm{asinh}\left(\frac{\vert J_0\bs{\hat{x}}\vert}{\vert h\vert}\right)}{2\vert J_0\bs{\hat{x}}\vert^5}ds(\bs{\hat{x}}), \nonumber
\end{align}
as well as
\begin{align}
I_1(h) & = \sum_{i=1}^3\hat{s}_j\int_{\partial\widehat{T}_j-\bs{\hat{x}}_0}\psi''_0(\bs{\hat{x}})\frac{2h^3 - 2h^2\sqrt{\vert J_0\bs{\hat{x}}\vert^2 + h^2} + \vert J_0\bs{\hat{x}}\vert^2\sqrt{\vert J_0\bs{\hat{x}}\vert^2 + h^2}}{3\vert J_0\bs{\hat{x}}\vert^4}ds(\bs{\hat{x}}) \nonumber \eqvspp
& \hspace{-.5cm} + h\sum_{i=1}^4e_i\sum_{j=1}^3\hat{s}_j\int_{\partial\widehat{T}_j-\bs{\hat{x}}_0}\hat{x}_1^{4-i}\hat{x}_2^{i-1}\frac{\dsp\frac{3h^2\vert J_0\bs{\hat{x}}\vert + \vert J_0\bs{\hat{x}}\vert^3}{\sqrt{\vert J_0\bs{\hat{x}}\vert^2 + h^2}} - 3h^2\mrm{asinh}\left(\frac{\vert J_0\bs{\hat{x}}\vert}{h}\right)}{2\vert J_0\bs{\hat{x}}\vert^5}ds(\bs{\hat{x}}) \nonumber \eqvspp
& \hspace{-.5cm} + h^2\sum_{i=1}^5f_i\sum_{j=1}^3\hat{s}_j\int_{\partial\widehat{T}_j-\bs{\hat{x}}_0}\hat{x}_1^{5-i}\hat{x}_2^{i-1}\frac{8h^4 + 12h^2\vert J_0\bs{\hat{x}}\vert^2 + 3\vert J_0\bs{\hat{x}}\vert^4 - 8h\left[\vert J_0\bs{\hat{x}}\vert^2 + h^2\right]^{\frac{3}{2}}}{3\vert J_0\bs{\hat{x}}\vert^6\left[\vert J_0\bs{\hat{x}}\vert^2 + h^2\right]^{\frac{3}{2}}}ds(\bs{\hat{x}}) \nonumber \eqvspp
& \hspace{-.5cm} + \sum_{i=1}^5g_i\sum_{j=1}^3\hat{s}_j\int_{\partial\widehat{T}_j-\bs{\hat{x}}_0}\hat{x}_1^{5-i}\hat{x}_2^{i-1}\frac{-8h^4 - 4h^2\vert J_0\bs{\hat{x}}\vert^2 + \vert J_0\bs{\hat{x}}\vert^4 + 8h^3\sqrt{\vert J_0\bs{\hat{x}}\vert^2 + h^2}}{3\vert J_0\bs{\hat{x}}\vert^6\sqrt{\vert J_0\bs{\hat{x}}\vert^2 + h^2}}ds(\bs{\hat{x}}) \label{eq:I1_vtx} \eqvspp
& \hspace{-.5cm} + \sum_{i=1}^7h_i\sum_{j=1}^3\hat{s}_j\int_{\partial\widehat{T}_j-\bs{\hat{x}}_0}\hat{x}_1^{7-i}\hat{x}_2^{i-1}\frac{-16h^6 - 24h^4\vert J_0\bs{\hat{x}}\vert^2 - 6h^2\vert J_0\bs{\hat{x}}\vert^4 + \vert J_0\bs{\hat{x}}\vert^6}{3\vert J_0\bs{\hat{x}}\vert^8\left[\vert J_0\bs{\hat{x}}\vert^2 + h^2\right]^{\frac{3}{2}}}ds(\bs{\hat{x}}) \nonumber \eqvspp
& \hspace{-.5cm} + \sum_{i=1}^7h_i\sum_{j=1}^3\hat{s}_j\int_{\partial\widehat{T}_j-\bs{\hat{x}}_0}\hat{x}_1^{7-i}\hat{x}_2^{i-1}\frac{16h^5\sqrt{\vert J_0\bs{\hat{x}}\vert^2 + h^2} + 16h^3\vert J_0\bs{\hat{x}}\vert^2\sqrt{\vert J_0\bs{\hat{x}}\vert^2 + h^2}}{3\vert J_0\bs{\hat{x}}\vert^8\left[\vert J_0\bs{\hat{x}}\vert^2 + h^2\right]^{\frac{3}{2}}}ds(\bs{\hat{x}}). \nonumber
\end{align}
The integrals \cref{eq:I0_vtx}--\cref{eq:I1_vtx} are, again, analytic but nearly singular. The main difference with \cref{eq:In1_param_h}, however, is that when the origin lies on an edge, these are no longer singular but $\OO(1)$ and $\OO(\delta\hat{x})$. As a consequence, the Gauss quadrature error with a single quadrature point is small. For example, for $\cref{eq:I0_vtx}$, the quadrature error when the origin is close to an edge plateaus at $\OO(\epsilon\log(2/\epsilon))$ as the distance to that edge $\epsilon\to0$. Moreover, since the corresponding integral is multiplied by $\epsilon$, this only generates an error $\OO(\epsilon^2\log(2/\epsilon))$ in the total computation of \cref{eq:I0_vtx}. Similarly, for $\cref{eq:I1_vtx}$, the Gauss quadrature error is $\OO(\epsilon^2\log(2/\epsilon))$ and introduces an error $\OO(\epsilon^3\log(2/\epsilon))$. 

We recommend using the transplanted Gauss quadrature rule of \cref{sec:step5} for both $\cref{eq:I0_vtx}$ and $\cref{eq:I1_vtx}$; it outperforms Gauss quadrature in theory and practice (see \cref{thm:convergence} and \cref{fig:exp2}).

\newpage
\section{Python code}\label{sec:code} 
The following short Python code assumes that \texttt{numpy} has been imported as \texttt{np} and uses the \texttt{minimize} method from \texttt{scipy.optimize}. Finally, \texttt{confmap(t,mu,nu)} returns the value of the map \cref{eq:g_munu} and of its derivative at points $t$ and for parameters $\mu$ and $\nu$.

\begin{figure}[H]
\begin{footnotesize}
\begin{verbatim}
# Step 1 - Mapping back:
a, b, c = 0.6, 0.7, 0.5                
Fx = lambda x: x[0] + 2*(2*a-1)*x[0]*x[1]
Fy = lambda x: x[1] + 2*(2*b-1)*x[0]*x[1]
Fz = lambda x: 4*c*x[0]*x[1]
F = lambda x: np.array([Fx(x), Fy(x), Fz(x)])                            # map 
J1 = lambda x: np.array([1 + 2*(2*a-1)*x[1], 2*(2*b-1)*x[1], 4*c*x[1]])  # Jacobian (1st col)
J2 = lambda x: np.array([2*(2*a-1)*x[0], 1 + 2*(2*b-1)*x[0], 4*c*x[0]])  # Jacobian (2nd col)
x0 = F([0.5, 1e-4]) + 1e-4*np.array([0, 0, 1])                           # singularity

# Step 2 - Locating the singularity:
e = lambda x: F(x) - x0                                            
E = lambda x: np.linalg.norm(e(x))**2                                    # cost function
dE = lambda x: 2*np.array([e(x) @ J1(x), e(x) @ J2(x)])                  # gradient
x0h = minimize(E, np.zeros(2), method='BFGS', jac=dE, tol=1e-12).x       # minimization
h = np.linalg.norm(F(x0h) - x0)                                       

# Step 3 - Taylor & 2D Gauss quadrature:  
n = 10; t, w = np.polynomial.legendre.leggauss(n)                        # 1D wts/pts
W = 1/8*np.outer(w*(1+t), w)                                             # 2D wts
X = np.array([1/2*np.outer(1-t, np.ones(n)), 1/4*np.outer(1+t, 1-t)])    # 2D pts
psi = lambda x: np.linalg.norm(np.cross(J1(x), J2(x), axis=0), axis=0)
tmp = lambda x,i: F(x)[i] - x0[i]
nrm = lambda x: np.sqrt(sum(tmp(x,i)**2 for i in range(3)))
tmp0 = lambda x,i: J1(x0h)[i]*(x[0]-x0h[0]) + J2(x0h)[i]*(x[1]-x0h[1])
nrm0 = lambda x: np.sqrt(sum(tmp0(x,i)**2 for i in range(3)))
f = lambda x: psi(x)/nrm(x) - psi(x0h)/nrm0(x)                           # regularized integrand
I = np.sum(W * f(X))                                                     # 2D Gauss
    
# Steps 4 & 5 - Continuation & 1D (transplanted) Gauss quadrature:
s1, s2, s3 = x0h[1], np.sqrt(2)/2*(1-x0h[0]-x0h[1]), x0h[0]              # Distances
dr1, dr2, dr3 = 1/2, np.sqrt(2)/2, 1/2
tmp = lambda t,r,i: (J1(x0h)[i]*r(t)[0] + J2(x0h)[i]*r(t)[1])**2
g, dg = confmap(t, -1 + 2*x0h[0], 2*s1)                       
          
r = lambda t: np.array([-x0h[0] + (t+1)/2,  -x0h[1]])                    # edge r1
nrm = lambda t: np.sqrt(tmp(t,r,0) + tmp(t,r,1) + tmp(t,r,2))
f = lambda t: (np.sqrt(nrm(t)**2 + h**2) - h)/nrm(t)**2
I += psi(x0h) * s1 * dr1 * (dg * w @ f(g))                               # 1D transplanted Gauss

r = lambda t: np.array([1 - x0h[0] - (t+1)/2, -x0h[1] + (t+1)/2])        # edge r2
nrm = lambda t: np.sqrt(tmp(t,r,0) + tmp(t,r,1) + tmp(t,r,2))
f = lambda t: (np.sqrt(nrm(t)**2 + h**2) - h)/nrm(t)**2
I += psi(x0h) * s2 * dr2 * (w @ f(t))                                    # 1D Gauss

r = lambda t: np.array([-x0h[0], 1 - x0h[1] - (t+1)/2])                  # edge r3
nrm = lambda t: np.sqrt(tmp(t,r,0) + tmp(t,r,1) + tmp(t,r,2))
f = lambda t: (np.sqrt(nrm(t)**2 + h**2) - h)/nrm(t)**2
I += psi(x0h) * s3 * dr3 * (w @ f(t))                                    # 1D Gauss
\end{verbatim}
\end{footnotesize}
\caption{\textit{Python code for computing the integrals of the first example when the origin is close to $\bs{\hat{r}}_1$. A code for general quadratic triangles/singularities is available on the first author's \href{https://github.com/Hadrien-Montanelli/singintpy}{GitHub}~page.}}
\label{fig:code}
\end{figure}

\section*{Acknowledgments}

We thank Christian Soize for sharing a classified manuscript that served as the basis of a series of papers in the 1990s \cite{angelini1992a, angelini1992b, angelini1992c}. We also thank Mikael Slevinsky for a fruitful exchange of emails about the computation of near-singular integrals using conformal mapping.

\bibliographystyle{siamplain}
\bibliography{/Users/montanelli/Dropbox/HM/WORK/ACADEMIA/BIBLIOGRAPHY/references.bib}

\end{document}